\documentclass[12pt]{article}
\usepackage{cite}
\usepackage{amsfonts,amsmath,upref,amssymb,amsthm,amsxtra,amscd,enumerate}
\usepackage[mathscr]{eucal}
\usepackage[a4paper,left=2cm,right=2cm,top=2cm,bottom=4cm]{geometry}
\usepackage{color,graphicx}
\input gpgmatharma.sty

\usepackage{tikz}\usetikzlibrary{patterns}  

\DeclareMathOperator{\diver}{div}
\DeclareMathOperator{\loc}{loc}

\newcommand{\dx}{\,\mathrm{d}x}

\def\Xint#1{\mathchoice
{\XXint\displaystyle\textstyle{#1}}%
{\XXint\textstyle\scriptstyle{#1}}%
{\XXint\scriptstyle\scriptscriptstyle{#1}}%
{\XXint\scriptscriptstyle\scriptscriptstyle{#1}}%
\!\int}
\def\XXint#1#2#3{{\setbox0=\hbox{$#1{#2#3}{\int}$ }
\vcenter{\hbox{$#2#3$ }}\kern-.6\wd0}}

\def\dashint{\Xint-}

\newcommand{\pat}{\partial_t}
\newcommand{\vu}{{\bfu}}
\newcommand{\vn}{{\bfn}}
\newcommand{\vv}{{\bfv}}
\newcommand{\ve}{{\bfe}}
\newcommand{\vg}{{\bfg}}

\newcommand{\vl}{{\bf l}}

\newcommand{\Pe}{\mathcal P}

\newtheorem{theorem}{Theorem}[section]
\newtheorem{lemma}[theorem]{Lemma}
\newtheorem{definition}[theorem]{Definition}

\newtheorem{remark}[theorem]{Remark}

\numberwithin{equation}{section}

\newcommand{\Dt}{D_t}
\newcommand{\dsx}{{\rm d}S_x}
\newcommand{\vr}{\varrho}
\newcommand{\vw}{\omega}
\newcommand{\vrh}{\vr_h}
\newcommand{\vrhkup}{\vrh^{k, \rm up}}

\newcommand{\vuh}{\vu_h}

\newcommand{\vuB}{\vu_{B}}
\newcommand{\bfvh}{\bfv_h}
\newcommand{\vgh}{\vg_h}
\newcommand{\vwh}{\vw_h}

\newcommand{\calah}{\calc_h}
\newcommand{\Qh}{Q_h}
\newcommand{\Vh}{\mathbf{V}_h}
\newcommand{\Vhz}{\mathbf{V}_{0,h}}
\newcommand{\intO}[1]{\int_{\mathscr B+\calc} #1 \dx}
\newcommand{\intC}[1]{\int_{\calc} #1 \dx}
\newcommand{\intB}[1]{\int_{\mathscr B} #1 \dx}
\newcommand{\TS}{\Delta t}
\newcommand{\jump}[1]{\left[\left[ #1 \right]\right]}
\newcommand{\edges}{\mathcal{F}_h}

\title{On the Motion of a Pendulum with a Cavity\\ Filled with a Compressible Fluid}
\date{}
\author{G. P. Galdi\footnote{University of Pittsburgh, email: galdi@pitt.edu} \ and V. M\' acha\footnote{Institute of Mathematics of the Academy of Sciences of the Czech Republic,
email: macha@math.cas.cz} \ and \v S. Ne\v casov\'a\footnote{Institute of Mathematics of the Academy of Sciences of the Czech Republic, email: matus@math.cas.cz} \ and B. She\footnote{Institute of Mathematics of the Academy of Sciences of the Czech Republic and Department of Mathematical Analysis of the Charles University in Prague, email:  she@math.cas.cz}}
\usepackage{graphicx}
\graphicspath{%
    {converted_graphics/}
    {C:/Users/galdi/OneDrive/Desktop/Necasova_Macha_Pendulum/Comments/}
    {/}
}
\begin{document} \pagenumbering{arabic}
\maketitle

\begin{abstract} We study the motion of the coupled system, $\mathscr S$, constituted by a physical pendulum, $\mathscr B$, with an interior cavity entirely filled with a viscous, compressible fluid, $\mathscr F$. The presence of the fluid may strongly affect on the motion of $\mathscr B$. In fact, we prove that, under appropriate assumptions, the fluid acts as a damper, namely,  $\mathscr S$ must eventually reach a rest-state. Such a state is  characterized by a suitable 
time-independent density distribution of $\mathscr F$ and a corresponding equilibrium position of the center of mass of $\mathscr S$.  These results are proved in the very general class of weak solutions and do not require any restriction on the initial data, other than having a finite energy.  We complement our findings with some numerical tests. The latter show, among other things,  the interesting property that ``large" compressibility  favors the damping effect, since it drastically reduces the time that $\mathscr S$ takes to go to rest.    
\end{abstract}
\section{Introduction}
The general problem of the motion of a rigid body with an interior, hollow cavity entirely filled with a fluid has all along attracted the attention of engineers and applied mathematicians. The list of major contributions only would be too long to include here, and for this we refer the reader to the monographs \cite{Ce0,Ce} and the references therein.\par One of the  remarkable phenomena that motivated this study traces back to the famous experiments  of {\sc Lord Kelvin} \cite{Tho}.  His tests unequivocally showed  that the presence of the fluid in the cavity  substantially influences the motion of the body by producing a significant  stabilizing effect. Modern primary applications of this distinctive property are, for example,  liquid sloshing dampers for vibration control of tall buildings \cite{EUS} and  oscillations suppressors in spacecraft and artificial satellites \cite{AlSp}.
\par
In spite of its relevance, a rigorous and systematic mathematical analysis of the motion of a body with a fluid-filled cavity has started only a few years ago \cite{DGMZ,GaIM,GaMa,GaMa1,GMZ,GMM,GMM1,Ma1,Ma2,Ma3}. These works have, on the one hand,  produced a full explanation of experimental observations and, on the other hand,  hinted at other, new interesting features that might be supported by numerical or lab tests. In particular, { a remarkable} result proved in \cite{GaMa,GMM1} shows that, under certain conditions, the presence of fluid can even bring the coupled system body-fluid  to  full rest.
\par
At this point, it must be emphasized that in all the papers indicated above, the fluid is supposed to be viscous and incompressible. Thus, more recently, in \cite{GaMaNe,GaMaNe2} we began to investigate the case where the fluid is still viscous but compressible. This study has a two-fold motivation. In the first place, to answer the natural question of the influence that compressibility may have on the characteristics of the terminal state. Secondly, the mathematical challenge constituted by the fact that, being the density no longer a constant, a much richer set of terminal states may occur and, therefore, the problem of their attainability can become of primary importance. In \cite{GaMaNe,GaMaNe2}, we limited ourselves to the problem where the coupled system body-fluid, $\mathscr S$, moves in absence of external forces (inertial motions). In particular, we proved, that for ``small" Mach numbers and for initial data of restricted magnitude, the system will reach a  terminal state where the body rotates with constant angular velocity and  the fluid is at rest with respect to the body. Notice that this result is in sharp contrast with the analogous one in absence of fluid, where the generic motion is a complicated {\em motion a la Poinsot}. This shows, in particular, the stabilizing effect of the fluid mentioned earlier on. 
\par
In the current article we begin to analyze the situation when the coupled system $\mathscr S$ is subject to external forces. To this end, we have chosen the classical example where the body is a physical pendulum whose interior is filled up with a viscous barotropic fluid with a classical constitutive law; see \eqref{eq:press.rel}. Our main findings will be described next. In the first place, we formulate the problem in the wide class of weak solutions, namely, suitably renormalized, distributional solutions satisfying the ``energy inequality" and corresponding to initial data that are only  requested to have a finite energy; see Definition \ref{def:weak}. Our objective is to investigate the behavior of these solutions as time goes to infinity and determine all possible terminal states. It comes then natural to consider the class of steady-state solutions, $\mathscr C$, as significant candidates. We thus show that, in such states, $\mathscr S$ must be at rest with a corresponding (time-independent) distribution of fluid density compatible with the vanishing of the axial component (that is, along the axis of rotation) of the total angular momentum. These states represent all  allowed equilibrium configurations for $\mathscr S$ and are characterized by having their center of mass, $C$, belonging to the vertical plane containing the axis of rotation; see Theorem \ref{th:char}. However, unlike the incompressible case, there could be more than two configurations of $\mathscr S$ that could furnish the same location of $C$, due to the fact that the density of the fluid is not constant, thus leading to the circumstance of multiple solutions; see Subsections \ref{RelCon} and \ref{RAU}. This fact makes the problem of attainability of steady-state solutions more complicated, also due to the lack of uniqueness of weak solutions. In any case, we are able to prove that, provided the cavity is convex, $\mathscr C$ is not empty, since it contains the non-empty class of minimizers of the total energy; see Theorem \ref{th:exi}. We then address the question of the asymptotic in time behavior in the class of weak solutions. While their existence can be obtained by a rather standard method (Theorem \ref{th:4.1}), their behavior for large times requires some efforts, especially for the proof of appropriate convergence of the pressure field; see Subsections \ref{Gloe} and \ref{lipre}. As a result, we are able to show that every weak solution tends to a steady state (equilibrium configuration), on condition that there is only one of them with total energy not greater than that of the initial data; see Theorem \ref{asbe}.  We then check that this condition is certainly satisfied if $\mathscr S$ possesses suitable symmetry properties. Precisely, we prove that if the cavity is a sphere with its center on the line passing through the center of mass of the body and its projection on the rotation axis, then whenever $\mathscr S$ is released from rest and in any position other than the straight-down and straight-up ones, it will eventually reach the equilibrium where its center of mass in the straight-down  position.        
\par
The above analytical findings are supported and complemented by several two-dimensional numerical tests. Here the coupled system $\mathscr S$ consists  of two concentric circles $C_1$ and $C_2\subset C_1$, where $C_1\backslash C_2$ is ``the body" and $C_2$ the ``cavity".  The objective is to study the behavior in time of $\mathscr S$, for different values of the physical quantities involved and, in particular, in the limit of very large values of the gas parameter $a$, that is, small Mach number (incompressible limit). The tests show, among other things, a surprising property, namely, that compressibility acts in favor of stability. In other words, all other parameters being fixed, $\mathscr S$ will reach the rest in a shorter time for ``large" $a$, rather than ``small" $a$.  
\par   
The plan of the paper is as follows. After formulating the problem in Section 2, including the definition of weak solution, in Section 2 we prove a characterization (Subsection 3.1) and the existence (Subsection 3.3) of steady-state solutions, along with some comments about their uniqueness (Subsections 3.2 and 3.4). Successively, in Section 4, we study the large-time behavior of weak solutions and prove there our main result on the attainability of steady states. Finally, Section 5 is dedicated to the numerical tests mentioned previously.

\section{Formulation of the Problem} 
Let $\mathscr B$ be a finite rigid body, with an interior hollow  cavity $\calc$ filled with a viscous fluid. In mathematical terms,
$\calc$ is  an
open simply connected domain of $\real^3$ completely surrounded by a domain $\mathscr B$  in such a way that $\partial\mathcal C\subset \mathscr B$, $\mathcal
C \cap \mathscr B = \emptyset$, and $\mathcal C \cup \mathscr B$ is bounded, simply connected and open.
\par
The body $\mathscr B $ is constrained at all times to rotate around a horizontal axis, ${\sf a}$, and we indicate by $O$ the orthogonal projection  of the center of mass $G$ of $\mathscr B$ on ${\sf a}$. Our objective is to study the motion of coupled system body-fluid and, in particular, its behavior for large times. To this end,
 let $\calf=\{O,\bfe_i\}$ be the fixed (inertial) frame  with $\bfe_3$ and $\bfe_1$ directed, respectively, along ${\sf a}$ and  the downward vertical, so that, indicating by $\hat{\bfg}$ the acceleration of gravity, in the frame $\calf$ we have 
\be
\hat{\bfg}=g\,\bfe_1\,,\ \ g=|\hat{\bfg}|.\label{g}
\ee
Further, let $\bfomega=\omega(t)\,\bfe_3$ be the angular velocity of $\mathscr B$, set
\be
\mathbb A(\omega):=\left(\ba{ccc} 0&-\omega&0\\ \omega &0&0\\
0&0&0\ea\right)\label{A}
\ee 
and denote by $\mathbb Q=\mathbb Q(t)$, $t\ge 0$, the family of proper orthogonal transformations solving the following IVP:
$$
\dot{\mathbb Q}=\mathbb A\cdot\mathbb Q\,,\ \ \mathbb Q(0)=\mathbb Q_0\,,
$$
with 
\begin{equation*}
\mathbb Q_0=\left(\ba{ccc} \cos \vartheta_0&-\sin \vartheta_0&0\\ \sin \vartheta_0 &\cos \vartheta_0&0\\
0&0&1\ea\right)\,,\ \ \mbox{some}\ \vartheta_0\in [0,2\pi).
\end{equation*}
Putting
$$
\vartheta(t):=\int_0^t\omega(s){\rm d}s+\vartheta_0\,,\ \ t\ge0\,,
$$
we obtain
\be
\mathbb Q(t)=\left(\ba{ccc} \cos \vartheta(t)&-\sin \vartheta(t)&0\\ \sin \vartheta(t) &\cos \vartheta(t)&0\\
0&0&1\ea\right)\,.\label{Q}
\ee
Let $\mathscr B_0$, $\calc_0$ be {\em arbitrarily fixed} reference configurations of $\mathscr B$ and $\calc$, respectively, and set, for all $t\ge 0$,
\be\ba{ll}\smallskip
\calc(t):=\{\bfy\in \real^3: \ \bfy=\mathbb Q(t)\cdot\bfx, \ \bfx\in\calc_0\}\,,\\ \smallskip
\mathscr B(t):=\{\bfy\in \real^3: \ \bfy=\mathbb Q(t)\cdot\bfx, \ \bfx\in\mathscr B_0\}\,,\\
\mathcal S(t):
=\mathscr B(t)\cup\calc(t)\,.
\ea\label{1}
\ee
Then, the equations of motion of the coupled system body-liquid in the frame $\mathcal F$ (that is, in the $y$-variable) are given by \cite{KNN}
\be\ba{cc}\smallskip\left.\ba{ll}\smallskip
\partial_t(r\bfw)+\Div(r\bfw\otimes\bfw)=\Div\hat{\mathbb S}(\bfw)-\nabla p(r)+r\,g\,
\bfe_1\\
\partial_tr+\Div(r\bfw)=0\ea\right\}\ \ (\bfy,t)\in\cup_{t>0}\,\calc(t)\times\{t\}\\\medskip
\bfw=\omega(t)\bfe_3\times\bfy\,\ \ (\bfy,t)\in\cup_{t>0}\,\partial\calc(t)\times\{t\}\\
\ode{}t\left(J(t)\omega +\bfe_3\cdot\Int{\calc(t)}{}r\, \bfy\times\bfw\,{\rm d}y\right)=\bfe_3\cdot\left[\left(\Int{\mathcal S(t)}{}\hat{r}\,\bfy\,{\rm d}y\right)\times g\bfe_1\right]\equiv g \Int{\mathcal S(t)}{}\hat{r}y_2\,{\rm d}y\,.
\ea\label{2}
\ee
Here, $r=r(\bfy,t)$, $\bfw=\bfw(\bfy,t)$ are density and velocity fields of the fluid, while $\hat{\mathbb S}$ is the viscous part of Cauchy stress tensor. Moreover, denoting by $r_{{}_\mathscr B}(\bfy)$ the density of $\mathscr B$, we set
$$
\hat{r}:=\left\{\ba{ll}\smallskip r(\bfy,t)\ &\ \mbox{if $y\in \calc(t)$}\\
\smallskip r_{{}_\mathscr B}(\bfy)\ &\ \mbox{if $y\in \mathscr B(t)$}
\ea\right.
$$  
and
$$
J(t):=\int_{\mathscr B(t)}r_{{}_\mathscr B}(\bfy)\delta^2(\bfy)\,{\rm d}y
$$
where $\delta(\bfy)=\dist (\bfy,{\sf a})$. 
\par
In order to convert the problem into an equivalent one where the domain of the fluid does not change with time, we define
\begin{equation*}\ba{ll}\smallskip
\varrho(\bfx,t):= r(\mathbb Q(t)\cdot\bfx,t),
\,  \bfu(\bfx,t):=\mathbb Q^\top(t)\cdot\bfw(\mathbb Q(t)\cdot\bfx,t), \  (\bfx,t)\in \calc_0\times (0,\infty),\\
\mathbb S(\bfu):=\mathbb Q^\top(t)\cdot\hat{\mathbb S}(\mathbb Q(t)\cdot\bfu)\cdot\mathbb Q(t), \, \bfg(t):=\mathbb Q^\top(t)\cdot\hat{\bfg},\,  \varrho_{{}_\mathscr B}(\bfx,t):= r_{{}_\mathscr B}(\mathbb Q(t)\cdot\bfx), \ t\in(0,\infty)\,,
\ea
\end{equation*}
so that, recalling that \be\mathbb Q(t)\cdot\bfe_3=\bfe_3\,,\ \ \mbox{ for all $t\ge 0$},\label{VFM}\ee the system (\ref{2}), in terms of the $\bfx$-variable and fields $\varrho$ and $\bfu$, thus becomes \cite{KNN}
\be\ba{cc}\smallskip\left.\ba{ll}\smallskip
\partial_t(\varrho\bfu)+\Div(\varrho\bfv\otimes\bfu)+\varrho\,\omega\bfe_3\times\bfu=\Div{\mathbb S}(\bfu)-\nabla p(\varrho)+\varrho\,\bfg\\
\partial_t\varrho+\Div(\varrho\bfv)=0\ea\right\}\ \ (\bfx,t)\in\calc_0\times(0,\infty)\\\medskip
\bfu=\omega(t)\bfe_3\times\bfx\,\ \ (\bfx,t)\in \partial\calc_0\times(0,\infty)\\
\ode{}t\left(I\,\omega +\bfe_3\cdot\Int{\calc_0}{}\varrho\, \bfx\times\bfu\, \dx \right)=\bfe_3\cdot\left[\left(\Int{\mathcal S_0}{}\hat{\varrho}\,\bfx\, \dx \right)\times \bfg\right]\,,
\ea\label{eq:compressible.NS}
\ee
where 
\be
\bfv:=\bfu-\omega\bfe_3\times\bfx\,, 
\label{eq:u.v.rel}
\ee
and
$$\ba{ll}\medskip
\hat{\varrho}:=\left\{\ba{ll}\smallskip \varrho(\bfx,t)\ &\ \mbox{if $x\in \calc_0$}\\
\smallskip \varrho_{{}_\mathscr B}(\bfx)\ &\ \mbox{if $x\in \mathscr B_0$}
\ea\right.\,,\\
I:=\Int{\mathscr B_0}{}\varrho_{{}_\mathscr B}(\bfx)\delta^2(\bfx)\, \dx \,.
\ea
$$
Moreover,  
$$
\mathbb S(\bfu) = 2\mu \mathbb D(\vu) + \left(\lambda-\frac 23\mu\right) {\mathbb I} \diver \bfu
$$
where $\mathbb D$ denotes the symmetric part of $\nabla\bfu$, $\mathbb I$ the identity matrix, while $\mu>0$ and $\lambda\ge 0$ are (constant) shear and bulk viscosity coefficients.
Also, observing that, by (\ref{A}) and (\ref{Q}),  $\dot{\mathbb Q}^\top\cdot\mathbb Q=\mathbb A(\omega)$. we derive
$$
\dot{\bfg}=\dot{\mathbb Q}^\top\cdot\mathbb Q\cdot\bfg=\mathbb A(\omega)\cdot\bfg\,,
$$
namely,
\begin{equation*}
\dot{\bfg}+\omega\,\bfe_3\times\bfg=\0\,,
 \ \  t\in(0,\infty)\,.
\end{equation*}
For the pressure $p$ we assume the following constitutive law
\begin{equation}\label{eq:press.rel}
p(\varrho) = a\varrho^\gamma\,,
\end{equation}
for some $a>0$ and $\gamma>3/2$.
Further, 
we endow (\ref{eq:compressible.NS}) with the initial conditions
$$
\varrho(0,\bfx) = \varrho_0(\bfx),\qquad \varrho(0,\bfx)\bfu(0,\bfx) = (\varrho\bfu)_0(\bfx)\,
$$
so that,  integrating
\eqref{eq:compressible.NS}$_2$ over $(0,t)\times\mathcal C_0$ for arbitrary $t\in \mathbb R$ we deduce the equation of conservation of mass for the fluid
\be\label{mcl}
\int_{\mathcal C_0} \varrho(t,\bfx) \dx  = \int_{\mathcal C_0} \varrho_0(\bfx)\dx.
\ee
The unknowns of \eqref{eq:compressible.NS} are $\bfu:(0,T)\times \calc_0\to \mathbb R^3$, $\varrho:(0,T)\times (\mathscr B_0\cup \overline \calc_0)\to \mathbb R$ and $\vg:(0,T)\to \mathbb R^3$, while we assume that the density $\varrho_{\mathscr B}$ of $\mathscr B$ is  prescribed.  However, instead of the unknown $\vu$, sometime we may find it more appropriate to use the velocity $\bfv$  defined in \eqref{eq:u.v.rel}.

If we formally multiply \eqref{eq:compressible.NS}$_1$ by $\vu$,  \eqref{eq:compressible.NS}$_4$ by $\omega $ and integrate by parts, we  deduce
the energy inequality:
\begin{equation*}\frac 12 \left(I\,{\rm d}_t|\omega|^2 + \pat \int_{\mathcal C_0} \varrho |\vu|^2\, \dx \right) + \int_{\mathcal C_0} \mathbb S(\bfv):\nabla\bfv\, \dx  + \pat \int_{\mathcal C_0}P(\varrho)\, \dx  \leq \pat \int_{\mathcal S_0}\hat{\varrho}\, \bfx\cdot \vg\, \dx \end{equation*}
which, after integration, leads to
\begin{equation*}\ba{ll}\medskip
\left[I\, \Frac{|\omega(t)|^2}{2} + \frac 12\Int{\mathcal C_0}{} \varrho(t) |\vu(t)|^2\dx + \Int{\mathcal C_0}{}P(\varrho(t))\dx - \Int{\mathcal S_0}{}\hat{\varrho}\, \bfx\cdot \vg\dx\right]_{t=0}^\tau \\
\hspace*{6.5cm} +  \Int0\tau \Int{\mathcal C_0}{} \mathbb S(\bfv):\nabla \bfv\dx{\rm d}t \leq 0,\ea
\end{equation*}
where $$P(\varrho) = \frac a{\gamma-1}\varrho^\gamma.$$

Our primary objective is to investigate the long-time behavior of  the system \eqref{eq:compressible.NS}--\eqref{eq:press.rel} in the class of weak solutions,  which we defined next.

\begin{definition}
A quadruple $(\varrho,\bfv,\omega,\vg)$ is a renormalized weak solution to \eqref{eq:compressible.NS} on time interval $(0,T)$ if~\footnote{Recall \eqref{eq:u.v.rel}\,.} \label{def:3.1}
\begin{itemize}
\item The momentum equation \eqref{eq:compressible.NS}$_1$ is fulfilled in a weak sense, i.e.
\begin{multline}
\int_0^T\int_{\mathcal C_0} \varrho \bfu\cdot \pat \bfphi \dx{\rm d}t + \int_0^T\int_{\mathcal C_0} \varrho \bfv \otimes \bfu: \nabla \bfphi\dx{\rm d}t - \int_0^T\int_{\mathcal C_0} \varrho \omega \ve_3 \times \bfu\cdot \bfphi \dx{\rm d}t\\ + \int_0^T\int_{\mathcal C_0}p(\varrho) \diver \bfphi \dx{\rm d}t - \int_0^T\int_{\mathcal C_0} \mathbb S(\bfu) :\nabla \bfphi \dx{\rm d}t\\ = -\int_0^T\int_{\mathcal C_0} \varrho \,\bfg\cdot \bfphi \dx{\rm d}t - \int_{\mathcal C_0} (\varrho\bfu)_0\cdot\bfphi(0)\dx{\rm d}t\label{eq:momentum.weak}
\end{multline}
for all $\bfphi\in C_c^\infty([0,T)\times \mathcal C_0)$, $\bfphi|_{\partial \mathcal C_0} = 0$.
\item The continuity equation is fulfilled in a renormalized weak sense, i.e.
\begin{multline}\label{eq:continuity.weak}
\int_0^T\int_{\mathcal C_0} b(\varrho) \pat \varphi \dx{\rm d}t  + \int_0^T\int_{\mathcal C_0} b(\varrho) \bfv \cdot\nabla \varphi \dx{\rm d}t\\ + \int_0^T\int_{\mathcal C_0} (b(\varrho) - b'(\varrho)\varrho)\diver \bfv\, \varphi \dx{\rm d}t = -\int_{\mathcal C_0} \varrho_0 \varphi(0)\dx
\end{multline}
for all $\varphi \in C^\infty_c([0,T)\times \overline{\mathcal C_0})$ and any $b\in C^1[0,\infty)$, $|b'(z)z|\leq c \sqrt{|z|}$.
\item The equations \eqref{eq:compressible.NS}$_{3,4}$ are fulfilled.
\item 
The energy inequality
$$
\mathcal E(\varrho(\tau),\vu(\tau),\omega(\tau), \vg(\tau)) \leq  \mathcal E\left(\varrho_0,\frac{(\varrho \bfu)_0}{\varrho_0},\omega_0,\vg_0\right) - \int_0^\tau \int_{\mathcal C_0} \mathbb S(\bfv):\nabla \bfv \dx{\rm d}t$$
is fulfilled for almost all $\tau \in [0,T)$, where
\be
\begin{array}{rl}\medskip\label{EnEr}
\mathcal E(\varrho,\vu,\omega,\vg) :=&\!\!\!
 I\, \displaystyle{\frac{|\omega|^2}{2} + \frac 12\int_{\mathcal C_0} \varrho |\vu|^2\dx + \int_{\mathcal C_0}P(\varrho) - P'(\overline \varrho) (\varrho - \overline\varrho) - P(\overline\varrho)\dx}
 \\&\!\!\! \displaystyle{- \int_{\mathcal S_0}\hat{\varrho} \bfx\cdot \vg\dx\,,}
\end{array}
\ee
and  $\overline\varrho = \frac 1{|\calc_0|} \int_{\mathcal C_0} \varrho \dx$ is constant in time due to \eqref{mcl}.
\end{itemize}\label{def:weak}
\end{definition}

\begin{remark} As shown in \cite{GaMaNe},  a sufficiently smooth weak solution defined as above solves, in fact, \eqref{eq:compressible.NS} pointwise.
\end{remark}

\begin{remark}
Our definition of weak solution allows us to deduce a weak formulation for a larger class of test functions. 
Precisely, take  
\begin{equation}\label{eq:test.f}
\bfphi = \bfphi_0 + \eta \ve_3\times \bfx\ \mbox{ where }\bfphi_0\in C^\infty_c([0,T)\times\mathcal C_0)\ \mbox{and } \eta\in C^\infty_c([0,T)).
\end{equation}
We multiply \eqref{eq:compressible.NS}$_4$ by $\eta$ to get
\begin{equation}\label{eq:bdry.term}
I \dot{\omega} \,\eta - \int_{\calc_0} \eta\,\pat(\varrho \bfu )\times \bfx \cdot\ve_3 \dx = \int_{\mathcal S_0} \eta\,\hat{\varrho} \bfx \dx \times \bfg \cdot\ve_3\, .
\end{equation}
Since, by definition,
\begin{equation*}
I\,  \ve_3 = \int_{\mathscr B} \varrho_\mathscr B \bfx \times (\ve_3 \times \bfx)\dx\,,
\end{equation*}
the term on the left hand side of \eqref{eq:bdry.term} is equal to
\begin{equation*}
\int_{\cals_0} \pat \eta (\hat{\varrho} \bfu) \cdot(\ve_3\times \bfx) \dx\,,
\end{equation*}
where $\bfu$ is extended to $\omega\bfe_3\times\bfx$ on $\mathscr B_0$.
Further, the right hand side of \eqref{eq:bdry.term} can be rewritten as
\begin{equation*}
\int_{\mathcal S_0} \eta\,\hat{\varrho} \vg \cdot (\ve_3\times \bfx)\dx\,,
\end{equation*}
so that, \eqref{eq:bdry.term} can be equivalently formulated as follows
\begin{equation}\label{eq:bdry.term.2}
\pat \int_{\cals_0} \eta(\hat{\varrho} \bfu ) \cdot(\ve_3\times \bfx)\dx - \int_{\cals_0} \pat \eta\,(\hat{\varrho} \bfu)\cdot(\ve_3\times \bfx)\dx = \int_{\cals_0} \eta\, \hat{\varrho}\, \vg \cdot(\ve_3 \times \bfx)\dx
\end{equation}
Next, we observe that
\begin{multline*}
\int_{\cals_0} \eta\, \hat{\varrho} \omega \ve_3  \times \bfu \cdot (\ve_3 \times \bfx)\dx = \int_{\cals_0} \eta\,\hat{\varrho}\,  \omega \ve_3\times(\bfv + \omega \ve_3\times \bfx)\cdot (\ve_3 \times \bfx)\dx 
\\= \int_{\cals_0} \eta\,\hat{\varrho} (\ve_3 \times \bfv)\cdot (\omega \ve_3 \times \bfx) \dx
 = \int_{\calc_0}\eta\,{ \varrho}  (v_2,-v_1,0)\cdot (\omega \ve_3 \times \bfx) \dx.
\end{multline*}
Furthermore,
\begin{multline*}
\int_{\cals_0} \hat{\varrho} (\bfv\otimes \bfu):\nabla(\eta \ve_3\times \bfx)\dx = \int_{\calc_0} \varrho \eta (\bfv \otimes \bfu):\nabla(\ve_3\times \bfx)\, \dx   = \int_{\calc_0} \varrho  \eta v_iu_j \partial_i \left(x_2,-x_1,0\right)_j\,  \dx \\ = \int_{\calc_0} \varrho \eta (-v_1 u_2 + v_2 u_1) \,  \dx  
 = \int_{\calc_0} \varrho \eta (v_2,-v_1,0)\cdot(\omega \ve_3\times \bfx)\, \dx  \,,
\end{multline*}
from which we deduce
\begin{equation}\label{eq:bdry.term.3}
-\int_{\cals_0} \hat{\varrho} (\bfv \otimes \bfu): \nabla(\eta \ve_3\times \bfx)\,  \dx  + \int_{\cals_0} \hat{\varrho} \omega \ve_3\times \bfu \cdot (\eta \ve_3\times \bfx)\,  \dx  = 0
\end{equation}
Also, we observe that since $\mathbb S$ is a symmetric tensor whereas $\nabla (\ve_3 \times \bfx)$ is antisymmetric, we get 
\begin{equation*}
\int_{\calc_0} \mathbb S(\bfv):\nabla(\eta \ve_3 \times \bfx) \,  \dx  = 0.
\end{equation*}
Thus, adding \eqref{eq:momentum.weak} with $\bfphi = \bfphi_0$, \eqref{eq:bdry.term.2} and \eqref{eq:bdry.term.3}, we obtain
\begin{multline}\label{eq:more.test}
\pat\Int{\cals_0}{} \hat{\varrho} \bfu\cdot \bfphi\dx - \Int{\cals_0}{} \hat{\varrho} \bfu\cdot\pat \bfphi \dx - \Int{\cals_0}{} \hat{\varrho} (\bfv\otimes\bfu):\nabla \varphi\dx
\\ 
+ \Int{\cals_0}{} \hat{\varrho} \omega \ve_3\times \bfu\cdot \bfphi\dx + \Int{\cals_0}{} \mathbb S(\bfv):\nabla \bfphi \dx = \Int{\cals_0}{} \hat{\varrho} \vg\cdot \bfphi\dx,
\end{multline}
where $\bfphi$ is a test function of the form \eqref{eq:test.f}.
Note that the energy inequality may be deduced formally from \eqref{eq:more.test} by taking $\bfphi = \bfu$.
\end{remark}

\section{Steady states}
One may expect that, for sufficiently large times, the generic weak solution may approach some steady state (namely, a time-independent solution of \eqref{eq:compressible.NS}) in a suitable topology. This will be investigated in Section \ref{sec:3}. Therefore, the main goal of this section is to find and characterize all possible  steady states, in the class of renormalized weak solutions. Before performing this study, however, we would like to make some simple but important remarks concerning the class of irrotational solutions to \eqref{eq:compressible.NS}, that is, those for which $\omega(t)=0$ for all $t\ge 0$.
\par
From what we presented at the beginning of the previous section, 
in those motions where $\omega(t)\equiv 0$, we have 
\be\mathbb Q(t)=\mathbb Q_0\,,\  \mbox{for all $t\ge 0$}, 
\label{Qt}
\ee
implying that \be\bfy=\mathbb Q_0\cdot\bfx,
\label{c}
\ee and, moreover, 
\be\label{gpg}
\varrho= r(\mathbb Q_0\cdot\bfx),\, \bfu=\mathbb Q_0^\top\cdot\bfw(\mathbb Q_0\cdot\bfx), \, \varrho_{{}_\mathscr B}= r_{{}_\mathscr B}(\mathbb Q_0\cdot\bfx), \, \bfg=\mathbb Q^\top_0\cdot\hat{\bfg}\,.
\ee
We notice that, by \eqref{1} and \eqref{Qt}, in such a case the position of the body (as well as  that of the cavity) is {\em time independent in the original frame $\mathcal F$}. We also notice that the system of equations \eqref{2} (or, equivalently \eqref{eq:compressible.NS})   { might seem overdetermined}. However, this is {\em not} the case, because, in general, we { cannot} expect that  motions with $\omega(t)\equiv 0$ may occur for {\em any} $\mathbb Q_0$ (that is, any orientation of $\mathscr B$). Therefore, $\mathbb Q_0$ (namely, $\vartheta_0$) becomes a {\em further unknown}, which thus makes the problem well-defined. 
\par
\subsection{Characterization}
With these premises, we now turn to the characterization of steady-state solutions.
From \eqref{eq:compressible.NS} we derive that they must satisfy  the following set of equations
\begin{equation}\label{eq:compressible.steady}
\begin{split}
\diver(\varrho\bfv\otimes\bfu) + \varrho\omega\ve_3\times\bfu + \nabla p(\varrho) - \diver \mathbb S(\bfu) & = \varrho\,\bfg\\
\diver(\varrho\bfv) & = 0\\
\bfu& = \omega\ve_3\times \bfx\\
\ve_3\cdot\left(\int_{\cals_0} \hat{\varrho} \bfx\dx\right)\times \vg &= 0\\
\omega \ve_3\times\vg & = 0.
\end{split}
\end{equation}
We work with a { \em renormalized weak solution}, i.e. a quadruple $(\varrho, \vv,\omega,\vg)$ (recall $\bfv = \vu -  \omega\ve_3\times \bfx$) which satisfies \eqref{eq:compressible.steady} and
$$
\diver (b(\varrho)\bfv) + (\varrho b'(\varrho) - b(\varrho))\diver \bfv = 0,\ \forall \; b\in C^1({\mathbb R})
$$
in distributional sense. 
The system \eqref{eq:compressible.steady} is complemented with the conservation of mass \eqref{mcl}:
\begin{equation}\label{eq:compressible.second}
\int_{\calc_0} {\varrho}(\bfx) \dx  = \int_{\calc_0} \varrho_0 (\bfx)\dx:=M\,.
\end{equation}
We also recall that the  gravity has prescribed magnitude, i.e.,
\begin{equation}\label{eq:compressible.grav}
|\vg| = |\vg_0|\equiv g\, .
\end{equation}
Now, since $\bfg\cdot\ve_3=0$, from the last equation in \eqref{eq:compressible.steady} we get $\omega = 0$, and so, arguing exactly  as in \cite[Lemma 1]{GaMaNe}, we show that $\bfv = 0$.
Consequently, \eqref{eq:compressible.steady} reduces to a system of only two relevant equations:
\begin{equation}\label{eq:steady.2}
\begin{split}
\nabla p(\varrho(\bfx)) & = \varrho(\bfx)\vg\ \ \mbox{in} \ \calc_0,\\
\bfe_3\cdot\left(\int_{\cals_0} \hat{\varrho}(\bfx) \bfx\dx\right)\times \vg& =0\,.
\end{split}
\end{equation}
Since $\omega=0$, by what we just proved and what we remarked at the beginning of this section, by \eqref{Qt}--\eqref{gpg} we deduce 
\be
\varrho:=r_s(\mathbb Q_0\cdot\bfx),\ \ \hat{\varrho}:=\hat{r}_s(\mathbb Q_0\cdot\bfx), \ \ \bfu\equiv\bfw\equiv 0, \ \
\bfg=\mathbb Q^\top_0\cdot\hat{\bfg}\,,
\label{s}
\ee 
for some $\mathbb Q_0$ {\em to be found}, where, from 
 \eqref{eq:steady.2}$_1$,  $r_s$ satisfies
$$
\nabla_{\mbox{\footnotesize $\bfx$}} p(r_s(\mathbb Q_0\cdot\bfx))=r_s(\mathbb Q_0\cdot\bfx)\,\bfg,
$$
and where we have emphasized that the derivatives are taken with respect to the $\bfx$-variable. 
Employing \eqref{c},  \eqref{s} and \eqref{g} in the latter, we show (derivatives now taken with respect to the $\bfy$-variable)
$$
\mathbb Q_0^\top\cdot[\nabla_{\mbox{\footnotesize$\bfy$}} p(r_s(\bfy))-r_s(\bfy)\,g\,\bfe_1]=0\,,
$$
which, recalling that $p(r_s)=a\,r_s^\gamma$, is in turn equivalent to 
$$
\ode{r_s^{\gamma-1}}{y_1}=\frac{\gamma-1}{a\,\gamma}\,g\,.
$$
Integrating both sides of this equation, and assuming that the cavity is convex we conclude
\be
r_s=r_s(y_1)=\left[\left(\frac{a\,\gamma}{\gamma-1}\,g\,y_1+c\right)_+
\right]^{\frac1{\gamma-1}}\,,\label{91}
\ee
for some constant $c\in\real$.
\par
We next investigate the class of all possible $\mathbb Q_0$ compatible with steady-state solutions, that is, the equilibrium configurations of the pendulum. From \eqref{eq:steady.2}$_2$ and \eqref{s} we obtain
\be
\bfe_3\cdot\left[\Int{\mathcal S_0}{}\hat{r}(\mathbb Q_0\cdot\,\bfx) \, \bfx\times \bfg\, \dx \right]=0\,.\label{A1}
\ee 
In this integral we now perform the change of variable \eqref{c}. Thus, taking into account \eqref{s}$_4$, \eqref{VFM}, \eqref{91} and that
$$
\bfx\times\bfg=(\mathbb Q_0^\top\cdot\bfy)\times (\mathbb Q_0^\top\cdot\hat{\bfg})=\mathbb Q_0^\top\cdot(\bfy\times\hat{\bfg})\,,
$$
we show that \eqref{A1} is equivalent to
\be
\bfe_3\cdot\left[\Int{\mathcal S_{eq}}{}\hat{r}(\bfy) \, \bfy\times \hat{\bfg}\,{\rm d}y\right]=0\,,\label{15}
\ee
where
\be\ba{ll}\smallskip
\mathcal S_{eq}:=\mathscr B_{eq}\cup 
\mathcal C_{eq}\,,\\ \smallskip
\mathscr B_{eq}:=\{\bfy\in \real^3: \ \bfy=\mathbb Q_0\cdot\bfx, \ \bfx\in\mathscr B_0\}\,,\\
\mathcal C_{eq}:=\{\bfy\in \real^3: \ \bfy=\mathbb Q_0\cdot\bfx, \ \bfx\in\mathcal C_0\}\,.
\ea\label{equ}
\ee
The relation \eqref{15} expresses the vanishing of the axial component of the angular moment of the coupled system $\mathscr S$ at equilibrium in the fixed frame $\mathcal F$.
 By keeping in mind    \eqref{91} and \eqref{g},  we show that \eqref{15}, in turn, is equivalent to the following one
\begin{equation}
\Int{\mathcal C_{eq}}{}r_s(y_1)y_2\,{\rm d}y+\Int{\mathscr B_{eq}}{}r_{{}_\mathscr B}(\bfy)y_2\,{\rm d}y=0\,,\label{16}
\end{equation}
which tells us that the center of mass $C$ of the coupled system must belong to the vertical plane containing the rotation axis ${\sf a}$.
Now, in any such equilibrium configurations $\mathcal S_{eq}$, the position $\bfy_C$ $(\equiv \overrightarrow{OC})$  in the fixed (inertial) frame is given by
$$
\bfy_C=\frac1{\calm}\left(\int_{\mathcal C_{eq}}r_s(y_1)\bfy\,{\rm d}y+\int_{\mathscr B_{eq}}r_{{}_\mathscr B}(\bfy)\,\bfy\,{\rm d}y\right)\,,
$$
with $\mathcal M=M+m$, and $m$ mass of $\mathscr B$. Since we chose $O$  as the orthogonal projection of the center of mass $G$ of $\mathscr B$ on ${\sf a}\equiv\bfe_3$, we have
$$
\int_{\mathscr B_{eq}}r_{{}_\mathscr B}(\bfy)\,
y_3\,{\rm d}y=0\,.
$$
Therefore, collecting the above results we conclude with the following characterization of steady states.
 
\begin{theorem} \label{th:char}
The quadruple $(\varrho_s,\bfu_s:=\bfv_s+\omega_s\bfe_3\times\bfx,\omega_s,\bfg_s)$ is a renormalized weak solution to \eqref{eq:compressible.steady} if and only if the following conditions {\rm (i)--(iii)} are met:
\begin{itemize}
\item[{\rm (i)}] $\omega_s=0$, $\bfu_s\equiv\0$. 
\end{itemize}
Setting $\bfy=\mathbb Q_0\cdot\bfx$: 
\begin{itemize}
\item[{\rm (ii)}]   $\varrho_s=r_s(y_1)$, where $r_s(y_1)$ is given in \eqref{91}, and $\hat{\bfg}_s=g\bfe_1$\,; 
\item[{\rm (iii)}] The rotation matrix $\mathbb Q_0$ is determined by the request that the center of mass $C$ of $\cals$ is located in the vertical plane, ${\sf V}$, containing the rotation axis ${\sf a}$, and precisely  at the point $\bfy_C$ that, in the fixed (inertial) frame, is given by  
\be
\bfy_C= \frac1{\calm}\left[\left(\int_{\mathcal C_{eq}}r_s(y_1)y_1\,{\rm d}y+\int_{\mathscr B_{eq}}r_{{}_\mathscr B}(\bfy)\,y_1\,{\rm d}y\right)\bfe_1+ \left(\int_{\mathcal C_{eq}}r_s(y_1)y_3\,{\rm d}y\right)\bfe_3\right]\,,\label{Nap}
\ee
with $\mathscr B_{eq}$ and $\calc_{eq}$ given in \eqref{equ}.\end{itemize}
Finally,
the constant $c$ in \eqref{91} is obtained by the condition
\be\label{eq:mass}
\int_{\mathcal C_{eq}}r_s(y_1)\,{\rm d}y=M\,.
\ee   
\end{theorem}
\subsection{Some relevant consequences of Theorem \ref{th:char}}\label{RelCon}
We would like to analyze some interesting   conclusions that can be drawn as corollary to the previous theorem. 
\par We begin to observe that
from \eqref{Nap} it follows that, in general,  in the equilibrium configuration, $\bfy_C\equiv\overrightarrow{OC}$ is {\em not} aligned with $\hat{\bfg}$ (namely, $\bfe_1$).  In fact, this alignment  occurs if and only if
\be
\int_{\mathcal C_{eq}}r_s(y_1)y_3\,{\rm d}y=0\,. 
\label{10}
\ee
The validity or invalidity of \eqref{10} may depend  on the location of the cavity with respect to $O$ (or, equivalently, $G$) and its shape. In particular, \eqref{10} holds if $G$ and $\mathcal C_{eq}$ are such that $(\cdot,y_3)\in\mathcal C_{eq}\, \Rightarrow\, (\cdot,-y_3)\in\mathcal C_{eq}$, but may not hold otherwise. A simple example is shown in the following figure.
\smallskip\par
\centerline{\includegraphics[width=3.49in,height=1.37in,keepaspectratio]{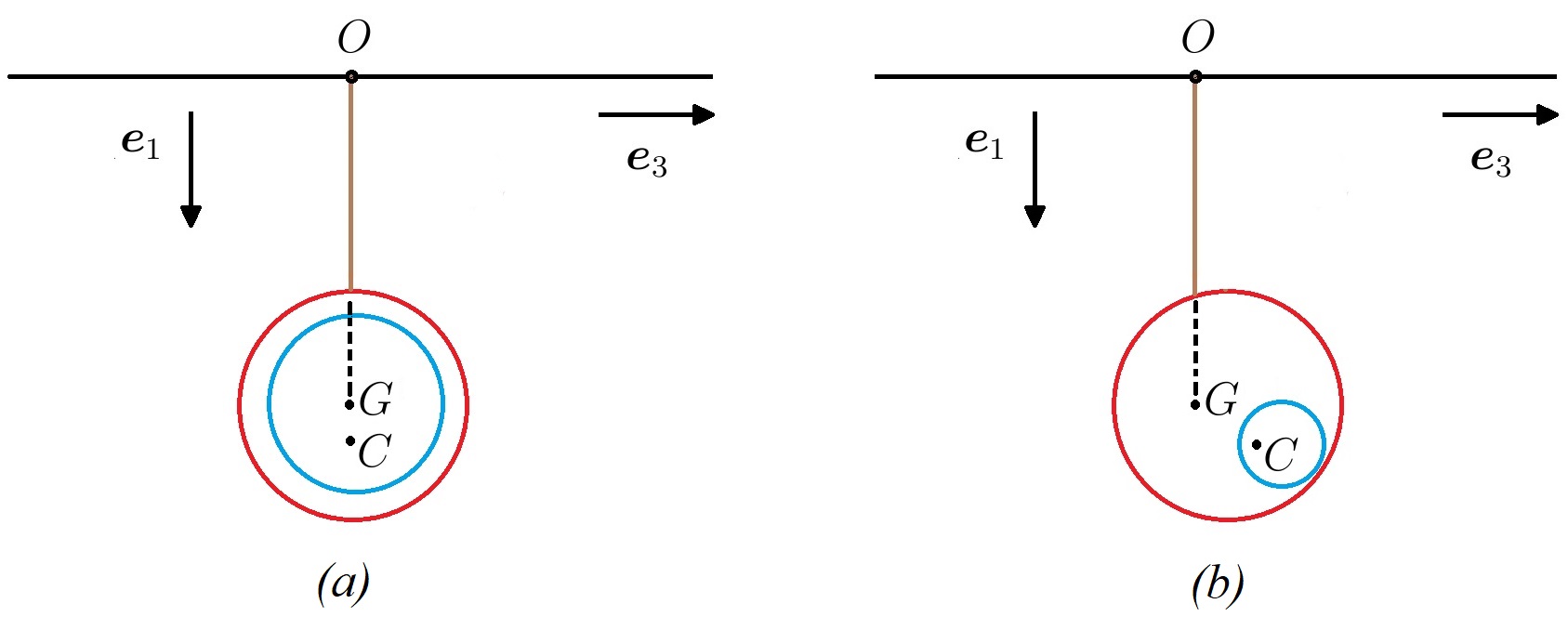}}
\begin{quote}{\rm \small Figure 1.: {\rm\footnotesize  {\em (a)} $\mathscr B$ is a homogeneous spherical shell with the cavity (blue) being the inner sphere. In this case $\bfy_C$ and $ \hat{\bfg}$ are parallel with same orientation; {\em (b)} $\mathscr B$ is a homogeneous sphere with an off-centered interior spherical cavity (blue). $\bfy_C$ and $\hat{\bfg}$ are not parallel.}}
\end{quote}
\smallskip\par
We shall next consider the case when the system $\mathscr S$ possesses some suitable symmetries.  Precisely, let $G\neq O$, and suppose  the cavity  $\calc$ is a body of revolution around the axis ${\sf e}:=\overrightarrow{OG}/|\overrightarrow{OG}|$. Moreover,  denote by $\alpha\in [0,2\pi)$ the angle between ${\sf e}$ and $\hat{\bfg}$. 
\smallskip\par
{\rm (1)} \,Consider the configurations of $\mathscr B$ (in the inertial frame) where ${\sf e}$ is parallel to $\hat{\bfg}$, namely, ${\sf e}=\pm\bfe_1$ corresponding to $\alpha=0,\pi$. Clearly, these are (the only two) equilibrium configurations, $\mathscr B_{eq}^{\pm}$, for $\mathscr B$, since we have
\be
\int_{\mathscr B_{eq}^{\pm}} r_\mathscr B(\bfy)y_i\,{\rm d}y=0\,,\ \ i=2,3\,.
\label{Ex1}
\ee
However, when $\mathscr B$ is in either one of the configurations $ \mathscr B_{eq}^{\pm} $,  the coupled system $\mathscr S$ is in equilibrium as well. In fact, denoted by $\calc_{eq}^{\pm}$ the positions of $\calc$ when $\mathscr B=\mathscr B_{eq}^{\pm}$,  owing to  the symmetry properties of $\calc$, we deduce
\be
\int_{\calc_{eq}^{\pm}} r_s(y_1)y_i\,{\rm d}y=0\,,\ \ i=2,3\,.
\label{Ex2}\ee 
Therefore, \eqref{Nap} follows from  \eqref{Ex1} and \eqref{Ex2}.\smallskip \par
{\rm (2)} \,Let us comment about the possible uniqueness of the above ``vertical" configurations. To this end,
denote by ${\sf G}={\sf G}(\alpha)$ the center of mass of $\calc$ with the density distribution given in \eqref{91}. In view of { Theorem \ref{th:char}(iii)},  equilibrium configurations for $\alpha\in (0,2\pi)$ may exist if and only if $C(\alpha)\in{\sf V}$, namely,
\be\label{wa}
{\sf w}_\alpha:=[M\,\overrightarrow{O{\sf G}}(\alpha)+m_{\mathscr B}\, \overrightarrow{OG}]\cdot\bfe_2=0\,,
\ee 
with $m_{\mathscr B}$ mass of $\mathscr B$. 
Since ${\sf w}_{0}={\sf w}_{\pi}=0$, this suggests that, for suitable $r_s$ (that is, $a,\gamma$ and $M$)  and $\calc$, the location of ${\sf G}(\alpha)$ may vary with $\alpha$ in such a way  that \eqref{wa} is satisfied also for $\alpha$ close to $0,\pi$, thus entailing the existence of some other equilibria,  around $\alpha=0,\pi$; see { Remark} \ref{ex:nonun}.  However, the latter circumstance is ruled out if $\calc$ has suitable symmetry.
 For example, assume that $\calc$ is a ball  centered at $O'$, and set ${\sf R}=|O'O|$.  Because $r_s=r_s(y_1)$ and of the  symmetry properties of $\calc$, it follows that ${\sf G}(\alpha)$ belongs to the straight line  parallel to $\bfe_1$ and passing through $O'$, for any $\alpha\in [0,2\pi)$. So, denoting by $\{O',\bfe_1'\}$ the frame with origin at $O'$ and $\bfe_1'$ parallel to and oriented as $\bfe_1$, we have $\overrightarrow{O'{\sf G}}=\ell'\bfe_1'$ ($\equiv \ell'\bfe_1$), for some $\ell'\in \real$. Notice that $\ell'$ is independent of $\alpha$. Setting $\ell:=|OG|$, we then infer
$$
\overrightarrow{O{\sf G}}=\overrightarrow{O{O'}}+\overrightarrow{O'{\sf G}}=({\sf R}\cos\alpha+\ell')\bfe_1-{\sf R}\sin\alpha\bfe_2\,,\ \ \overrightarrow{O{G}}=\ell(\cos\alpha\bfe_1-\sin\alpha\bfe_2)\,,
$$
and, consequently,  condition \eqref{wa} becomes
\be\label{eq:eq}
(M\,{\sf R}+m_{\mathscr B}\,{\ell})\sin\alpha=0\,,
\ee
that is satisfied if and only if $\alpha=0,\pi$, which means that  
 the  configurations discussed in {\rm (1)} are the {\em only possible equilibria} for $\mathscr S$. \smallskip\par
{\rm (3)} \,If $G$ lies outside the cavity $\calc$, there exists a ``critical angle," $\alpha_c>0$, such that no equilibrium is allowed for $\alpha\in(\alpha_c,\pi-\alpha_c)$.  In fact, let $\Gamma$ be the smallest cone having vertex at $G$ and containing $\calc$. Then, on the one hand, $C$ must belong to the segment $\overline{G{\sf G}}$ which, on the other hand, must be in the interior of $\Gamma$ or, at most, overlap with one of its generatrices. Thus, $\alpha_c$ is precisely the least value of $\alpha$ for which $\Gamma\cap {\sf V}=\emptyset$.
\smallskip\par
Suppose now $G=O$ and let $\calc$ be a ball centered at $O$. Then, clearly,
\be\label{RV}
\int_{\mathscr B_{eq}}r_{\mathscr B}(\bfy)\bfy\,{\rm d}y=0\,, 
\ee
for all $\mathscr B_{eq}$. Moreover, by symmetry,
\be\label{RV1}
\int_{\mathcal C_{eq}}r_{s}(y_1)y_i\,{\rm d}y=0\,,\ \ i=2,3\,,\ee
for all $\mathcal C_{eq}$. From \eqref{RV} and \eqref{RV1} we deduce that \eqref{16} is satisfied in  
every position of $\mathscr S$, implying the existence of a continuum of equilibrium configurations.
\medskip\par
Throughout this section, we have provided explicit examples of solutions to the steady-state problem \eqref{eq:steady.2}, as a consequence of their characterization furnished in Theorem \ref{th:char}. However, at this stage, we do not know if  \eqref{eq:compressible.second}--\eqref{eq:steady.2} admits a solution for every given $\mathscr B$ and $\mathcal C$. The (positive) answer to this question will be furnished in the following subsection.
\subsection{Existence}

Objective of this subsection is to show existence to the problem  \eqref{eq:compressible.second}--\eqref{eq:steady.2}. The main question to address here is not to find the distribution of density  (this was already done in the previous subsection) but, instead,  to provide the existence of  an orientation of $\mathscr S$ with respect to $\bfg$ compatible with a steady state ({\em equilibrium configurations}) or, equivalently, the matrix $\mathbb Q_0$ introduced in the previous section.

To reach this objective, we notice, as before,  that the first equation in \eqref{eq:steady.2}  entails
\begin{equation}\label{eq:density.relation}
\varrho = \sqrt[\gamma-1]{\left(\frac{\gamma-1}{a\gamma}  \bfx \cdot \vg  + c\right)_+}
\end{equation}
for some constant $c\in \real$.  We also recall that  we are assuming $\calc$  convex, and thus $\supp \varrho$ has just one connected component;  see \cite{FePe2}. 

Denote by $\mathcal P$ the projection $\mathbb R^3\to\mathbb R^3$, $\mathcal P:(x_1,x_2,x_3)\mapsto (x_1,x_2,0)$. The second equation in \eqref{eq:steady.2} then yields that $\vg$ is parallel to $\mathcal P\left(\int_{\mathcal S} \hat{\varrho}(\bfx){\bfx}\,  \dx \right)$.\footnote{For simplicity of notation, in what follows we set $\mathscr S_0\equiv\mathscr S$, $\mathscr B\equiv\mathscr B_0$, and $\calc\equiv\calc_0$.}  Thus, setting
$${\bf l} := \mathcal P\left(\int_{\mathscr B} \varrho_{\mathscr B}(\bfx) \bfx \,  \dx \right)\,,
$$
we infer that \eqref{eq:compressible.second}--\eqref{eq:steady.2} is equivalent to the following system of four equations
\begin{equation}\label{eq:steady.ready}
\begin{split}
\int_{\calc}\sqrt[\gamma-1]{\left(\frac{\gamma-1}{a\gamma}  \bfx \cdot \vg  + c\right)_+}\,  \dx  - \int_{\calc} \varrho_0\,  \dx  & = 0\\
d \vg - \Pe \left(\int_{\calc} \sqrt[\gamma-1]{\left(\frac{\gamma-1}{a\gamma}  \bfx \cdot \vg  + c\right)_+} \bfx \,  \dx \right) - {\bf l}& = 0\\
|\vg|^2 - |\vg_0|^2 &= 0
\end{split}
\end{equation}
for four unknowns: $\vg = (g_1,g_2,0)\in \mathbb R^3$, $ c\in \mathbb R$, and $d\in \mathbb R$. 

In view of the above and of what established in the previous subsection,  we can then state the following lemma.
\begin{lemma}\label{lemm3.5}
Let $\vv$, $\varrho$, $\omega$ and $\vg$ be a renormalized weak solution to \eqref{eq:compressible.steady}. Then
$\vv = 0,\ \omega = 0, \ \varrho$ is given by \eqref{eq:density.relation}
and $\vg$ satisfies \eqref{eq:steady.ready}.
\end{lemma}
\begin{remark}
We would like to explain the meaning of the parameter $d$. Equation \eqref{eq:steady.ready}$_2$ can be rearranged as
$$
d{\bf g} = \mathcal P\left(\int_{\mathcal C} \sqrt[\gamma-1]{\left(\frac{\gamma-1}{2\gamma}  \bfx \cdot \vg  + c\right)_+} \bfx \dx \right) + {\bf l}.
$$
As we know from previous subsection (see \eqref{Nap}), the right-hand side of this equation is the vector $\overrightarrow{OC}$, with $C$ center of mass of the whole system at equilibrium. Thus, $d\neq 0$ means that $\overrightarrow{OC}$ and $\bfg$ must be parallel. Moreover, $d$ positive means that the $\bfg$ and $\overrightarrow{OC}$ have the same orientation ($C$ is below the hinge), whereas $d$ negative means the opposite ($C$ is above the hinge).
\end{remark}

By using the standard theory associated to the Euler--Lagrange equations, we can show the following result.

\begin{lemma}\label{lemm3.6}
Let $\varrho_s\in L^\gamma(\calc)$ and $\vg_s\in \mathbb R^2\times \{0\}$ be a minimizer of the functional
$$
\mathcal I: (\varrho, \vg) \mapsto \mathcal E (\varrho, 0,0,\vg)
$$
with $\mathcal I$  defined on the set $$\{\varrho\in L^\gamma, \vg \in \mathbb R^2\times \{0\},\ \varrho \geq 0,\ |\vg| = |\vg_0|,\ \int_{\calc}\varrho(\bfx)\,  \dx  = \int_\calc \varrho_0(\bfx)\,  \dx  \}.$$
Then $\varrho_s$ and $\vg_s$ solve the system \eqref{eq:density.relation} and \eqref{eq:steady.ready}.
\end{lemma}

\begin{proof}
In particular, for fixed $\varrho_s$, $\vg_s$ is a minimizer of the smooth functional
$$
\mathcal I_{\varrho_s}:\vg\mapsto \mathcal I(\varrho_s,\vg)
$$
where, for simplicity, we assume $\vg\in \mathbb R^2$ as only the first and second components of $\vg$ matter. We also assume that $\vg$ satisfies the constraint \eqref{eq:compressible.grav}. 
Then the standard results for Lagrangian multipliers yields
$$
- \mathcal P\left(\int_{\mathcal S} \hat{\varrho}(\bfx) \bfx\,  \dx \right) - 2\lambda \vg_s = 0 
$$
for some $\lambda\in \mathbb R$. We get \eqref{eq:steady.ready}$_2$ assuming $\varrho$ is given by \eqref{eq:density.relation} and $d = -2 \lambda$.
Likewise, for fixed $\vg_s$, $\varrho_s$ is  a minimizer of the functional
$$
\mathcal I_{{\footnotesize \vg_s}}:\varrho \mapsto \mathcal I(\varrho,\vg_s)
$$
where $\varrho$ ranges in the nonempty closed convex set 
$$
K:=\left\{\varrho\in L^\gamma,\ \int_{\mathcal C}\varrho\,  \dx  = \int_{\mathcal C}\varrho_0\,  \dx ,\ \varrho\geq 0\ \mbox{on }\mathcal C\right\}.
$$
According to \cite[Corollary 2.184 $\&$ Example 2.186]{CaLeMo},
the minimizer $\varrho_s$ satisfies
\begin{equation}\label{eq:stationary.point}
0\ni \partial \mathcal I_{\vg_s}(\varrho_s) + N_K(\varrho_s)
\end{equation}
where $\partial$ denotes Frech\'et derivative and $N_K(\varrho_s)$ is the  normal cone defined by
$$
N_K(\varrho_s) := \left\{\eta\in L^{\gamma'},\ \forall f\in K\ \int_{\mathcal C} \eta(\bfx)(f(\bfx)-\varrho_0(\bfx))\,  \dx \leq 0\right\}.
$$
We now analyze the structure of $N_K$. First, let $\eta$ be a constant. Then
$$
\int_{\mathcal C} \eta(\bfx)(f(\bfx)-\varrho_s(\bfx))\,  \dx  = 0 \ \ \mbox{for all $f\in K$}\,,
$$
which implies that  every constant function belongs to $N_K$. Next, let $\eta\in N_K$, and let us show that  $\eta|_{\supp \varrho_s}$ must be  a constant. Suppose otherwise. Without loss of generality, we  may assume $\int_{\supp \varrho_s}\eta\dx =0$. Consequently, there exist $\varepsilon>0$ and sets $A, B\subset \supp \varrho_s$ of positive measure such that $\eta|_A>\varepsilon$, $\eta|_{B}<-\varepsilon$, $\varrho_s|_{A\cup B}>\varepsilon$ and $|A| = |B|$.
Take $f = \varrho_s + \varepsilon(\chi_A - \chi_B)$ where $\chi_S$ is the  characteristic function of the set $S$. Then 
$$
\int_{\mathcal C} \eta(\bfx) (f(\bfx)-\varrho(\bfx))\,  \dx  > 0\,,
$$
which yields that $\eta$ is not in $N_K$: a contraddiction. Summing up,  we can thus state that  every function in $N_K$ is constant on $\supp\varrho_s$ and every constant belongs to $N_K$.
Consequently, \eqref{eq:stationary.point} yields
$$
 P'(\varrho_s) - P'(\overline\varrho) - \bfx \cdot \vg_s + \lambda = 0\ \mbox{on }\supp \varrho_s\,,
$$
for some $\lambda\in \mathbb R$. Since $P'(\varrho_s) = \frac{a\gamma}{\gamma-1}\varrho^{\gamma-1}$, this equation implies \eqref{eq:density.relation} with $c = \frac{(P'(\overline\varrho) - \lambda)(\gamma-1)}{a\gamma}$.

\par\end{proof}
We are now in a position to show our existence result. 
\begin{theorem}\label{th:exi}
Suppose $\calc$ convex. Then \eqref{eq:compressible.second}--\eqref{eq:steady.2} has at least one solution.
\end{theorem}
\begin{proof}
In view of Lemmas \ref{lemm3.5} and \ref{lemm3.6}, we only have to prove that the functional 
$$
\varrho, \vg \mapsto \mathcal E (\varrho, 0,0,\vg)
$$
has at least one minimizer in the set
$$
\mathcal A := \{(\varrho,\vg)\in L^\gamma(\calc)\times (\mathbb R^2\times \{0\}), \varrho \geq 0, \int_\calc \varrho \,  \dx  = \int_\calc \varrho_0 \,  \dx , |\vg| = |\vg_0|  \}.
$$
We begin to show that for a fixed $\vg$, the functional 
$$
\mathcal I_{{\footnotesize \bfg}}:\varrho \mapsto \mathcal E (\varrho, 0, 0,\vg)
$$
defined on
$$
\cala_0:=\{\varrho\in L^\gamma(\calc), \varrho \geq 0, \int_\calc \varrho \dx = \int_\calc \varrho_0\dx\} 
$$
attains there a minimum. Consider the function
$$
P^\infty(z) = \left\{
\begin{array}{l}
P(z) - P'(\overline \varrho)(z-\overline \varrho) - P(\overline\varrho),\ \mbox{for }z\in [0,\infty),\\
+\infty,\ \mbox{for }z\in (-\infty, 0).
\end{array}
\right.
$$ 
and  redefine $\mathcal I_{\footnotesize \bfg}$ in the following way
$$
\mathcal I_{\footnotesize \bfg}:\varrho \mapsto \int_\calc \left( P^\infty(\varrho) + \varrho  \bfx \cdot \vg  \right) \,  \dx .
$$
Owing to  \cite[Theorem 6.54]{FoLe}, this functional is lower semicontinuous, and since $\cala_0$ is convex and closed, we obtain the  existence of a minimizer by the direct method of calculus of variations; see \cite[Section 3.2]{FoLe}.
Next, consider the function
$$
f:\vg \mapsto \min_{\varrho} \mathcal E(\varrho,0,0,\vg)\,,
$$
defined in  $E:=\{\vg\in \mathbb R^2\times \{0\}, |\vg| = |\vg_0|\}$ with values in $\mathbb R$. In order to show the theorem, it remains to prove that $f$  attains a minimum in $E$. Since $E$ is compact, it suffices to check that $f$ is continuous there. 
The definition of $\mathcal E$ yields 
$$
|\mathcal E(\varrho, 0 , 0, \vg_1)  -\mathcal E(\varrho, 0,0,\vg_2)| \leq c|\vg_1 - \vg_2|
$$
with $c$ independent of $\varrho$ (but dependent on $\varrho_0$), from which it follows that
$$
\min_\varrho \mathcal E(\varrho,0,0,\vg_1) \leq \min_\varrho \mathcal E(\varrho,0,0,\vg_2) + c |\vg_1 - \vg_2|.
$$
Interchanging the role of $\vg_1$ and $\vg_2$, we deduce the opposite inequality, which furnishes the desired continuity and thus completes the proof of the theorem.
\par\end{proof}
\subsection{Further comments about uniqueness}\label{RAU}
We shall now provide a result regarding the uniqueness of the ``vertical" equilibrium configurations that relates to what discussed in Section \ref{RelCon}(1). To this end, 
set
$$\Pi(\vg) = \mathcal P \left(\int_{\mathcal C} \sqrt[\gamma-1]{\left(\frac{\gamma-1}{a\gamma}  \bfx \cdot \vg  + c\right)_+} \bfx \,  \dx \right)$$ where the constant $c$ is determined uniquely by \eqref{eq:steady.ready}$_1$.
The following result holds.
\begin{theorem}\label{thm:uniq}
Assume  the cavity $\calc$ is such that, for any $\vg\in \mathbb R^2\times\{0\}$, $|\vg| = |\vg_0|$, 
\begin{equation}\label{ass1}
\left\langle \Pi(\vg) ,\vl \right\rangle > 0\,.
\end{equation}
Then, the corresponding $d$ is not 0. Assume, further, we are in a class of solutions such that
\begin{equation}\label{ass3}
|d|> \delta_2>0\,,
\end{equation}
and 
\begin{equation}\label{ass2}
|\Pi(\vg_1) - \Pi(\vg_2)|\leq \delta_1 |\vg_1 - \vg_2|\,,
\end{equation}
for some $\delta_2>2\delta_1>0$. Then,  there are at most two solutions to \eqref{eq:steady.ready}, one with $d<0$ and the other with $d>0$.
\end{theorem}
\begin{proof}
We begin to notice that  condition \eqref{ass1} is guaranteed once we know, for example, that every $\bfx\in \mathcal C$ satisfies $\langle \bfx,\vl\rangle >0$.
From \eqref{eq:steady.ready} we get 
$$
|d| = \frac{|\Pi(\vg) + \vl|}{|\vg_0|}\,,
$$
which, by \eqref{ass1}, implies $|d|>0$.
We distinguish the two cases   $d>0$ and $d<0$, and begin to treat  the case $d>0$ first. Let $\vg_1, c_1, d_1$ and $\vg_2, c_2, d_2$ be two solutions to \eqref{eq:steady.ready}. Employing \eqref{ass2}, we infer
\begin{equation*}
|d_1 - d_2| = \frac{||\Pi(\vg_1) + \vl| - |\Pi(\vg_2) + \vl||}{|\vg_0|}\leq \frac{|\Pi(\vg_1) - \Pi(\vg_2)|}{|\vg_0|}\leq \frac{\delta_1}{|\vg_0|} |\vg_1 - \vg_2|\,.
\end{equation*}
On the other hand, from \eqref{eq:steady.ready}$_2$ we show 
\begin{multline*}
0=\langle d_1 \vg_1 - d_2 \vg_2 - (\Pi(\vg_1) - \Pi(\vg_2)),\vg_1 - \vg_2\rangle = \\
 d_1 |\vg_1-\vg_2|^2 + (d_1-d_2)\langle \vg_2,\vg_1 - \vg_2\rangle - \langle \Pi(\vg_1) - \Pi(\vg_2), \vg_1 - \vg_2\rangle\\
\geq (\delta_2-2\delta_1) |\vg_1 - \vg_2|^2\geq (\delta_2 - 2\delta_1) |\vg_1- \vg_2|^2
\end{multline*}
and thus,  assuming $d>0$ and $\delta_2-2\delta_1>0$, there is at most one solution.
Note that the same conclusion holds also in the case $d<0$. \par\end{proof}
\begin{remark} The assumptions of Theorem \ref{thm:uniq} are rather significant. In fact, they ensure that the center of mass of the whole system does not vary too much for different directions of gravity. We now show that these assumptions are somehow also necessary, by bringing an example that shows that, if they are violated, the conclusion of the theorem is not  true. Let consider $\mathcal C = (-1,1)\times (-1,1)\times (-1,1)$, $\gamma=2$, and $a=\frac 12$. The total mass is assumed to be 4. Furthermore, the body is such that $\vl = (1,0,0)$. 
Then $\vg_1 = (1,0,0)$ and $\vg_2 = (-1,0,0)$ are two solutions for which the appropriate $c_1$ and $c_2$ is both equal to $1$. However, for $\vg_1$ we have
$$
\Pi(\vg_1) + \vl = (11/3,0,0) = \frac{11}3 \vg_1
$$
and for $\vg_2$ we have
$$
\Pi(\vg_2) + \vl = (-5/3,0,0) = \frac53 \vg_2
$$
and we have two solutions for which $d>0$ (namely, $d_1 = \frac{11}3$ and $d_2 = \frac53$). Notice that in this case, $\delta_1=2$, $|d|=11/3$, so that $|d|<2\delta_1$, and \eqref{ass2} is violated, for all $\delta_2>2\delta_1$.
\label{ex:nonun}
\end{remark}

\section{Global behavior of weak solutions}
\label{sec:3}
For simplicity, in what follow we set $\mathcal C_0\equiv\mathcal C$.
We begin by  stating an existence result of weak solutions.

\begin{theorem}\label{th:4.1}
Let $\mathcal C$ be of class $C^{2+\nu}$, for some $\nu>0$, and let $\varrho_0\in L^\gamma(\mathcal C)$, $\gamma>3/2$, with $\varrho_0|_{\mathscr B} = \varrho_c$, $\varrho_c\in \mathbb R$. Further, let $\vu_0:\mathcal S\to \mathbb R^3$ be such that $\varrho_0|\vu_0|^2\in L^1(\mathcal C)$ and $\vu_0|_{\mathscr B} = \omega \ve_1\times {\bfx}$ for some $\omega\in \mathbb R$. Then there exists a weak solution to \eqref{eq:compressible.NS} in the sense of  {\rm Definition} \ref{def:3.1} on the time interval $(0,T)$, arbitrary $T>0$.
\end{theorem}

\begin{remark}
The proof of this theorem is omitted, since it can be obtained by simply combining the arguments used in \cite{FeNo} in the case when the motion of $\mathscr B$ is prescribed with  those of  \cite{GaMaNe2}, where the motion of $\mathscr B$ is a further unknown. 
The crucial point is to show  uniform estimates to  derive the regularity of the pressure by using the Bogovski operator that allows for the  passage to the limit  in the pressure term. A detailed treatment of this issue can be found in \cite{GaMaNe2} and \cite{FeNoPe}.
\end{remark}

\begin{remark}
The regularity on $\mathcal C$ stated in the theorem could  be relaxed  to assume $\calc$ to be just of class $C^{0,1}$ (or even less regular). The method may be found in \cite{kuku}.
\end{remark}

\subsection{Global estimates}\label{Gloe}
Hereinafter we assume that $\varrho$ and $\vu$ is a weak solution in the sense of Definition \ref{def:3.1}. Moreover, we assume $\mathcal C$ is of class $C^{0,1}$  and $\gamma > \frac 32$.
We recall that, from the energy inequality, we deduce the following estimates
\begin{equation*}
\begin{split}
\mbox{ess sup}_{t\in (0,\infty)}\|\varrho(t,\cdot)\|_{\gamma}&\leq c\\
\sup_{t\in (0,\infty)}|\omega(t)|&\leq c\\
\mbox{ess sup}_{t\in (0,\infty)}\|\varrho(t,\cdot) |\bfu|^2(t,\cdot)\|_1& \leq c\\
\|\vu\|_{L^2((0,\infty)\times \mathcal C)} & \leq c\\
\|\nabla\bfu\|_{L^2((0,\infty)\times \mathcal C)} & \leq c
\end{split}
\end{equation*}
for some $c>0$.
In order to perform the long-time analysis of our solutions we need some other uniform bounds that we are going to derive.   
First, we observe that from \eqref{eq:compressible.NS}$_5$ we get at once
\begin{equation}\label{eq:g.bound}
|\vg(t)| = |\vg(0)|\ \mbox{for all }t\in (0,\infty).
\end{equation}
We next define the sequence
$$
(\varrho_n(t),\bfv_n(t),\omega_n(t)):= (\varrho(n+t),\bfv(n+t),\omega(n+t))
$$
and investigate its behavior as $n\to\infty$.
Throughout, we shall use the letter $c$ to denote an arbitrary constant  independent of $n$.
We begin to show higher integrability properties of the density, by adapting a method from \cite[Section 7.9.5]{NoSt}. Consider the test function 
$$
\varphi(t,\bfx) = \psi(t)\Phi(t,\bfx),\ \Phi = {\mathfrak B} \left(S_\alpha(b_k(\varrho_n)) - \dashint_{\mathcal C} S_\alpha(b_k(\varrho_n))\, {\rm d}t\right)
$$
where $\psi\in C^\infty_c(-1,2)$, ${\mathfrak B}$ is the Bogovski operator,  $S_\alpha$ is a mollifying operator with respect to time and 
$$
b_k(\varrho) = \left\{ 
\begin{array}{l}
\varrho^\nu\ \mbox{for }\varrho\in [0,k)\\
k^\nu\ \mbox{for }\varrho \in [k,\infty)
\end{array}
\right.
$$
for some $\nu\in (0,\frac23\gamma-1]$. Such a $\varphi$ is an admissible test function for \eqref{eq:compressible.NS}. We thus obtain
\begin{multline*}
\int_{-1}^2 \psi\int_{\mathcal C} p(\varrho_n) S_\alpha(b_k(\varrho_n))\,  \dx {\rm d}t = \int_{-1}^2 \int_{\mathcal C} \psi p(\varrho_n)\left(\dashint_{\mathcal C} S_\alpha(b_k(\varrho_n)\right)\,  \dx {\rm d}t\\
+ \int_{-1}^2 \int_{\mathcal C} \psi  \mathbb S(\bfv_n):\nabla \Phi\,  \dx {\rm d}t + \int_{-1}^s\int_{\mathcal C} \psi    \varrho_n \vg \cdot \Phi \,  \dx {\rm d}t - \int_{-1}^2\int_{\mathcal C} \psi  \varrho_n \bfv_n \otimes \bfu_n : \nabla \Phi\,  \dx {\rm d}t\\
-\int_{-1}^2\int_{\mathcal C} \varrho_n \bfu_n \cdot \Phi \pat \psi \,  \dx {\rm d}t - \int_{-1}^2\int_{\mathcal C} \varrho_n \bfu_n \cdot \pat \Phi \psi\,  \dx {\rm d}t + \int_{-1}^2\int_{\mathcal C}\psi \varrho_n \omega_n \ve_3 \times \bfu_n \cdot \Phi \,  \dx {\rm d}t
\end{multline*}
Every term above, except for  the last one, may be estimated similarly as it is done in \cite[Section 7.9.5.2]{NoSt}. The last term may be estimated as follows (compare with the estimate of term $J_5$ in \cite[Section 7.9.5.2]{NoSt})
\begin{multline*}
\left|\int_{-1}^2\int_{\mathcal C}\varrho_n \omega_n \ve_3 \times \bfu_n \Phi \,  \dx {\rm d}t\right|\leq c\int_{-1}^2\int_{\mathcal C}|\psi|\varrho_n |\omega_n|^2 |\Phi|\,  \dx {\rm d}t + c\int_{-1}^2\int_{\mathcal C}|\psi| \varrho_n |\omega_n||\bfv_n| |\Phi| \,  \dx {\rm d}t\\
\leq c \|\psi\|_{L^1} \|S_\alpha(b_k(\varrho_n))\|_{L^\infty(L^{\frac{6\gamma}{5\gamma-3}})}.
\end{multline*}
Thus, one may let  $\alpha\to 0$ and $k\to\infty$ and, in the same fashion as \cite{NoSt}, to deduce
$$
\int_0^1\int_{\mathcal C}\varrho_n^{\gamma + \nu}\,  \dx {\rm d}t \leq c.
$$
Furthermore, from the energy and Korn's inequalities we easily derive
\begin{equation*}
\int_\tau^{\tau+1} \int_\calc |\nabla \bfv_n|^2\,  \dx {\rm d}t\to 0
\end{equation*}
as $\tau \to \infty$. 
As a result, along a subsequence,
\begin{equation}\label{eq:convergence}
\begin{split}
\varrho_n &\to \varrho_s \ \mbox{weakly in }L^{\gamma + \nu}((0,1)\times \calc)\\
\bfv_n& \to \bfv_s\equiv 0\ \mbox{weakly in }L^2(0,1; W^{1,2}(\calc))\\
\omega_n & \to \omega_s\ \mbox{weakly}^*\ \mbox{in }L^\infty(0,1)\\
p(\varrho_n) & \to p(\varrho)_s\ \mbox{weakly in } L^{1+ \gamma/\nu}((0,1)\times \calc).
\end{split}
\end{equation}
The functions $\varrho_s,\ \bfv_s,\ \omega_s$ and $p(\varrho)_s$ solve \eqref{eq:compressible.steady} and thus $\omega_s = 0$.
Notice that $p(\varrho)_s$ denotes a weak limit of $p(\varrho_n)$ and since $p$ is nonlinear, it is not necessarily true that $p(\varrho)_s = p(\varrho_s)$. We shall address this issue in the next subsection. 
\subsection{Limit of the pressure term}\label{lipre}

We will prove that $p(\varrho)_s = p(\varrho_s)$. To this end, it is sufficient to adapt the method from \cite[Section 4]{FePe}. Let
$$
G(z) = z^\alpha,\ 0<\alpha<\min\left\{\frac 1{2\gamma}, \frac{\nu}{2(\nu+\gamma)}\right\}
$$
and  consider a function $b(z) = G(p(z))$ in \eqref{eq:continuity.weak} to deduce
\begin{multline*}
|\langle \pat G(p(\varrho_n)),\varphi\rangle| =\\ \left| \int_0^1\int_{\mathcal C} G(p(\varrho_n))\bfv_n \nabla \varphi\,  \dx {\rm d}t + \int_0^1\int_{\mathcal C} (G(p(\varrho_n)) - G'(p(\varrho_n)) \varrho_n)\varphi \bfv\,  \dx {\rm d}t\right| 
\leq  c \|\varphi\|_{1,q_1},
\end{multline*}
for some $q_1>1$ and for $\varphi \in C^\infty_c((0,1)\times\mathcal C)$.
Consequently
$$
\mbox{Div}_{t,x} (G(p(\varrho_n),0,0,0)\ \mbox{is precompact in }W^{-1,q_1}_{\loc}((0,1)\times \calc).
$$
We know that 
$$
|\langle \nabla p(\varrho_n),\varphi\rangle|  = \left|-\int_0^1\int_{\mathcal C} p(\varrho_n)\diver \varphi\,  \dx {\rm d}t\right| \leq c \|\varphi\|_{1,q_2}.
$$
for some $q_2>1$ and for $\varphi \in C^\infty_c((0,1)\times\mathcal C)$. Thus
$$
\mbox{Curl}_{t,x} (p(\varrho_n),0,0,0)\ \mbox{is precompact in }W^{-1,q_2}_{\loc}((0,1)\times \calc)
$$
The well known div-curl lemma (see \cite{tartar}) yields 
\begin{equation}\label{eq:div.curl}
G(p(\varrho_n))p(\varrho_n)\to G(p(\varrho)_s)p(\varrho)_s.
\end{equation}
According to \cite[Theorem 6.2]{pedregal} there exists a parametrized family of probabilistic measures $\nu_{t,x}$ on $[0,\infty)$ such that
$$
\varrho_s(t,\bfx) = \int_0^\infty \rho \, {\rm d}\nu_{t,x}(\rho).
$$
and, according to \eqref{eq:div.curl}, we also have
\begin{equation}\label{eq:young}
\int_0^\infty \rho^{\alpha\gamma + \gamma}\, {\rm d}\nu_{t,x}(\rho) = \int_0^\infty \rho^{\alpha\gamma}\, {\rm d}\nu_{t,x}(\rho) \int_0^\infty \rho^\gamma \, {\rm d}\nu_{t,x}(\rho).
\end{equation}
where we assume for simplicity that $p(\varrho) = \varrho^\gamma$. 
Fix $(t,\bfx)$ and set $\theta^{\alpha\gamma} := \int_0^\infty \rho^{\alpha\gamma}\ {\rm d}\nu_{t,x}(\rho)$. Then \eqref{eq:young} yields
$$
\int_0^\infty \left( \rho^{\gamma\alpha + \gamma}- \theta^{\gamma\alpha} \rho^\gamma - \theta^\gamma(\theta^{\alpha\gamma} - \rho^{\alpha\gamma})\right)\, {\rm d}\nu_{t,x}(\rho) = 0,
$$
which transforms into
$$
\int_0^\infty (\rho^\gamma - \theta^\gamma)(\rho^{\alpha\gamma} - \theta^{\alpha\gamma})\, {\rm d}\nu_{t,x}(\rho) = 0.
$$
The integrand is strictly positive for all $ \rho \neq \theta$ and since $\nu_{t,x}$ is a probabilistic measure, we get $\nu_{t,x} = \delta_{\varrho_s(t,\bfx)}$ where $\delta_\alpha$ is a Dirac mass at point $\alpha$. Consequently, 
$$
\varrho_n\to \varrho_s\ \mbox{strongly in }L^{q}((0,1)\times\calc),
$$
for all $q\in [1,\gamma + \nu)$ yielding $p(\varrho)_s = p(\varrho_s)$.

\subsection{Large-time behavior}
In view of what we have proved in the previous subsections, we may now proceed to the limit in \eqref{eq:compressible.NS} and deduce that $\varrho_s,\ \omega_s,\ \vg_s$ solve \eqref{eq:compressible.steady}. Furthermore, \eqref{eq:convergence} allows us to pass to a limit also in the energy as follows
\begin{equation*}
\lim_{t_n\to\infty} \int_{t_n}^{t_n+1}\mathcal E(\varrho(t),\vu(t),\omega(t),\vg(t))\ {\rm d}t  = \mathcal E(\varrho_s,0,0,\vg_s)\,.
\end{equation*}
We will assume that  there is only one solution to \eqref{eq:steady.2}    
 fulfilling the condition
$$
\mathcal E(\varrho_s,0,0,\vg_s) \leq \mathcal E(\varrho_0,\vu_0,\omega_0,\vg_0)\, .
$$
In such a case, as there is only one possible limit, we immediately get $\vg(t)\to \vg_s$ as $t\to \infty$. Due to \eqref{eq:compressible.NS}$_2$ we have $\partial_t\varrho\in L^2(W^{-1,2})$ and, consequently
$
\varrho(t) \to \varrho_s
$
as $t\to \infty$.

We have just proved the following theorem.
\begin{theorem}\label{asbe}
Let $\calc$ be a Lipschitz domain and let the initial conditions $\varrho_0,\bfv_0,\omega_0$ and $\vg_0$ be the same as in Theorem \ref{th:4.1}. Assume that there is just one solution to \eqref{eq:compressible.steady} for which $\mathcal E(\varrho_s,\0,0, \vg_s) \leq \mathcal E(\varrho_0,\vu_0,\omega_0,\vg_0)$. Then every renormalized weak solution to \eqref{eq:compressible.NS} tends to $(\varrho_s,\0,0,\vg_s)$. More precisely,
\begin{equation*}
\begin{split}
\varrho(t)&\to \varrho_s \ \mbox{weakly in } L^\gamma\ \mbox{as }t\to \infty,\\
\bfv(t_n + t)&\to 0\ \mbox{strongly in  } L^2(0,1; W^{1,2}(\calc))\ \mbox{as}\ t_n\to \infty,\\
\omega(t) & \to 0 \mbox{ as }t\to \infty,\\
\vg(t) &\to \vg_s \ \mbox{ as } t\to \infty.
\end{split}
\end{equation*}
\end{theorem}
\par
As a simple application of this theorem, consider the case when the cavity $\calc$ is a sphere $S$ with its center $O'$ belonging to the straight line $\overline{OG}$, $G\neq O$. Then, from Section \ref{RelCon}(3), we know that there are two and only two equilibrium configurations, namely, with the pendulum either in the straight-down or straight-up position. More precisely, these configurations are characterized by two numbers $\sigma^+>0$ and $\sigma^-<0$, such that 
\be\label{OC}\overrightarrow{OC^\pm}=\sigma^\pm\,\bfe_1,\ee 
corresponding to the case when the center of mass $C$ of the coupled system $\mathscr S$ is below $(C^+)$ or above $(C^-)$ the hinge. Let us denote by $(\varrho_s^+,\bfg_s^+)$ and $(\varrho_s^-,\bfg_s^-)$ the two associated steady-state solutions, and set $\cale^\pm:=\cale(\varrho_s^\pm,\0,0, \vg_s^\pm)$.  Thus, $\varrho_s^+=\varrho_s^-\equiv r_s$, with $r_s$ given in \eqref{91}, and
$$
\int_{S^+}P(\varrho_s^+) - P'(\overline \varrho) (\varrho_s^+ - \overline\varrho) - P(\overline\varrho)\dx=\int_{S^-}P(\varrho_s^-) - P'(\overline \varrho) (\varrho_s^- - \overline\varrho) - P(\overline\varrho)\dx\,, 
$$
where $S^+$ [resp. $S_-$] denotes the position of the sphere in the straight-down [resp. straight-up] configuration of $\mathscr S$.
Moreover,
$$
\int_{\mathscr B}\hat{\varrho}\,\bfx\cdot\bfg=g\calm\,\overrightarrow{OC}\cdot\bfe_1\,.
$$
Collecting all the above, using \eqref{OC} and recalling \eqref{EnEr}, we show
\be
\mathcal E^+-\mathcal E^-=-g\,\calm\,(\sigma_++\sigma_-)<0\,.
\label{inda}\ee
Therefore, if we choose the initial data in such a way that
\be\label{inda1}
\cale^+<\mathcal E(\varrho_0,\vu_0,\omega_0,\vg_0)<\cale^-\,,
\ee
then {\em every} (renormalized) weak solution will converge for large times to $(\varrho_s^+,\bfg^+)$, namely, the pendulum will eventually reach the configuration with its center of mass in its lowest position. This will certainly happen, if we start the pendulum from rest ($\bfu_0\equiv{\bf0},\omega_0=0)$ and pick $(\varrho_0,\bfg_0)\neq (\varrho^\pm,\bfg^\pm)$, that is, the pendulum is initially away from either  straight-down and straight-up configurations.\footnote{Actually, if $\mathscr S$ is initially in one of these two positions with $\bfu_0$ and $\omega_0$ both vanishing, it will stay there for all times.}  In fact, from Lemma \ref{lemm3.6} we know that any minimizer of $\cale(\varrho,{\bf 0},0,\bfg)$ is a solution to \eqref{eq:compressible.second}--\eqref{eq:steady.2} and, by Theorem \ref{th:exi} that the set of minimizers is not empty. However, from the results of Section \ref{RelCon}(3) and \eqref{inda}, we deduce that $(\varrho^+,\bfg^+)$ is the only minimizer, which proves our claim.  
\section{Numerical results}
As we mentioned in the introductory section, in \cite{GaMa} a problem analogous to the one treated here was investigated under the assumption that the fluid filling the cavity was incompressible. One interesting point to investigate is whether there is any quantitative difference between the two problems. For example, how the characteristic time taken to reach the terminal state (the rest) depends on the compressibility of the fluid. Unfortunately, an analytic study of such a question is, to date, beyond our grasp. However, we have performed numerical tests that may suggest the answer. Objective of this section is to present these findings.  
\par 
For simplicity, we assume that the flow is two-dimensional -- this is a reasonable assumption as the physical phenomenon may hint to neglect the third dimension. 
\par
We propose a mixed finite volume -- finite element method for the approximation of the system \eqref{eq:compressible.NS} that we are going to describe next. 
\subsection{The mixed finite volume -- finite element scheme}
To begin, let $\calah$ be a regular and quasi-uniform triangulation of the cavity $\calc$ and $\edges$ be the set of all interior faces of $\calah$. Further, we write $h=\max_{K \in \calah} h_K$ as the mesh size, where $h_K$ is the diameter of an element $K\in \calah$. We denote by $\Qh$ the space of piecewise constant functions and by $\Vh$ the piecewise linear Crouzeix--Raviart element space:
\[ \Qh = \left\{v\in L^1(\calc) |\;  \bfv_K \text{ is a constant} \; \forall\; K \in \calah \right\},\]
\[
\Vh = \left\{\bfv\in L^2(\calc) \;|  \; \bfv_K \text{ is a piecewise affine function}\; \forall \; K \in \calah; \; \int_{\sigma} \jump{\bfv}\dsx=0\; \forall\; \sigma \in \edges \right\},
\]
where $\jump{\cdot}|_{\sigma}$ represents the jump over the interface $\sigma$. 
To specify the homogeneous Dirichlet boundary condition, we define 
\[\Vhz = \left\{\bfphi \in \Vh \; |\; \int_\sigma \bfphi \dsx =0 \; \forall \;  \sigma \in \partial \calc\right\}.
\]
Now we are ready to introduce the following mixed finite volume -- finite element method. 
\paragraph{Numerical method (compressible solver).} 
Let $\TS$ be the time increment, $\vgh^0=\vg(0)$, $\vwh^0=\vw(0)$, and let $(\vrh^0, \vuh^0)$ be the projection of the initial data $(\vr,\vu)(0)$ onto the space $\Qh \times \Vh$. 
Then, 
for $k=1,\dots,N_t=T/\TS$  
we seek $(\vrh^k, \vuh^k, \vgh^k, \vwh^k) \in \Qh \times \Vh \times \mathbb R^2 \times \mathbb R$ as solutions to the following system of algebraic equations 
\begin{subequations}\label{scheme}
\begin{equation}\label{scheme_g}
    \Dt \vgh^k  +  \vwh^{k-1} \ve_3 \times \vgh^{k+1/2} =0,
\end{equation}
where $\vgh^{k+1/2} = \frac{\vgh^{k-1} +\vgh^k}{2}$ and 
$\Dt v_h^k = \frac{v_h^k - v_h^{k-1}}{\TS}$; 
\begin{equation}\label{scheme_r}
   \int_K \Dt \vrh^k  \dx 
   + \int_{\partial K}\vrhkup \bfvh^{k-1} \cdot \vn \dsx
   =0 \text{ for all } K\in\calah\,,
\end{equation}
where
$ \bfvh= \vuh -\vuB, \vuB = \vwh \ve_3 \times \bfx$,   
$\vn$ is the outer normal vector, and $\vrh^{up}$ is the so-called upwind value of the density given by 
\begin{equation*} \vrh^{up} = \begin{cases}\lim_{\delta\to0}\vrh(\bfx+\delta \vn)& \text{if } \bfvh\cdot\vn \geq 0,\\
\lim_{\delta\to0}\vrh(\bfx - \delta \vn) & \text{otherwise};
\end{cases}
 \end{equation*}
\begin{multline}\label{scheme_u}
\frac12 \intC{ \Big( 
   \Dt (\vrh^k\vuh^k) \cdot \bfphi
   +\vrh^{k-1} \Dt \vuh^k \cdot \bfphi
   +\vrh^k\bfvh^{k-1}\cdot \nabla\vuh^k \cdot \bfphi 
   -\vrh^k\bfvh^{k-1}\cdot \nabla \bfphi \cdot \vuh^k 
   \Big) } 
\\
+ \intC{\vrh^k w_h^{k-1} \ve_3\times\vuh^k \cdot \bfphi }
  + \intC{ \Big( \mathbb S(\vuh^k): \nabla \bfphi 
  - p(\vrh^k)\diver \bfphi \Big)} 
  = \intC{\vrh^k \vgh^{k+1/2} \cdot \bfphi},
\\  \mbox{ for all }\; \bfphi \in \Vhz ;
\end{multline}
\begin{equation*}
  I_{B33} \Dt \vwh^k 
  + \Dt \left(\intC{\vrh^k \bfx \times \vuh^k} \right) \cdot \ve_3 
 =\intO{\vrh^k \bfx }\times \vgh^{k+1/2} \cdot \ve_3.
 \end{equation*}
\end{subequations}

\begin{remark}
The scheme \eqref{scheme} enjoys the following properties for all $k=1,\dots, N_t$: 
\begin{itemize}
\item Conservation of mass. Indeed, summing up over all elements leads to the  mass conservation. 
\[ \intC{\vrh^k} = \intC{\vrh^{k-1}}=\cdots=\intC{\vrh^{0}}.    \]
\item Conservation of gravity in the sense of \eqref{eq:g.bound}, i.e.,
\[ |\vgh^k| = |\vgh^{k-1}|=\cdots=|\vgh^{0}|, \]
which can be easily obtained by multiplying \eqref{scheme_g} with $\vgh^{k+1/2}=\frac{\vgh^{k-1} +\vgh^k}{2}$.
\item Positivity preserving of density.  We have $\vrh^k>0$ provided $\vrh^0>0$, for which we refer the proof to \cite[Lemma 4.1]{Karlsen}. 
\end{itemize}
\end{remark}

\subsection{Numerical experiments}\setcounter{figure}{1}
We take the pendulum as a circular plate with a circular cavity in the center 
\[
\mathscr B =\left\{\bfx \;| R_0 \leq \sqrt{(x_1-L)^2 +x_2^2} \leq R_1 \right\}, \quad 
\calc=\left\{\bfx\;|\sqrt{(x_1-L)^2 +x_2^2} \leq R_0 \right\}, 
\]
with $R_0=0.1$, $R_1=0.2$, and $L$ be the length of the pendulum, see Figure~\ref{fig:pen}. In our numerical experiments we set $\gamma=5/3$, $\mu=100$, and $\eta = 0$ if not otherwise mentioned. Further, we denote $\varrho_{\mathscr B}$ as the density of the body $\mathscr B$, $\overline{\vr_0} = \frac{1}{|\calc|} \intC{\vr(0)}$ as the averaged initial density of the fluid in the cavity $\calc$ and $R_{\vr}=\varrho_{\mathscr B}/ \overline{\varrho_0}$ as the ratio of the densities. The initial data are set as $\varrho(0)=1, \vu(0)={\bf 0}, \vw(0)=0, \vg(0) =(\cos{\vartheta_0},\sin{\vartheta_0})$ with $\vartheta_0 = \pi/45$. 
\begin{figure}[!h]
\centering
\begin{tikzpicture}[scale=1.]
\draw[fill=blue!20,very thick] (0,0) circle (1.2);
\draw[fill=white,very thick] (0,0) circle (0.6);
\draw[thick,->,black] (1.2,4.8)--(-0.4,-1.6) node[left] {$x_1$}; 
\draw[thick,->,black] (0,4.25)--(2,3.75) node[below] {$x_2$}; 
\fill[black] (1,4) circle (2pt);
\fill[red] (0,0) circle (1pt);
\draw[very thick,->,black] (0,0)--(0,-1.5) node[right] {$\vg$};
\draw[thick,dashed,black] (1,4)--(1,2); 

\path node at (0.2,2) { \large$L$};
\draw[thick,red,->] (0,0) -- (0.52,0.3); 
\path node at (0.27,-0.1) { \color{red} $R_0$};
\draw[thick,blue,->] (0,0) -- (-1.2,0); 
\path node at (-0.8,-0.25) { \color{blue} $R_1$};
\draw[] (1,4) -- (0.7,2.8) arc(249:270:0.9);
\path node at (0.8,2.5) { \large $\vartheta$};
\end{tikzpicture}
\caption{Pendulum with a cavity.}
\label{fig:pen}
\end{figure}
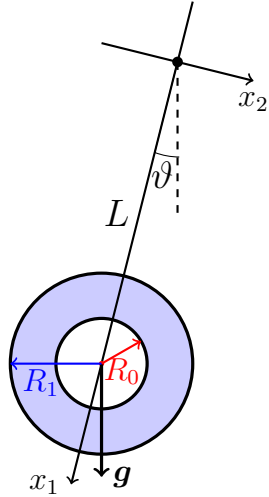

\subsubsection{Experiment 1: influence of gas parameters of the compressible solver~\eqref{scheme}}
We show in Figure~\ref{fig:test1} - Figure~\ref{fig:test3} the evolution of pendulum position (angle $\vartheta$) for different values of density ratio $R_{\vr}$, gas parameter $a$, and pendulum length $L$. First of all, in all these numerical experiments we observe the effect of the dissipation due to the viscosity of the fluid.
Moreover, we see \emph{larger} dissipation effects for: 
\begin{enumerate}
\item smaller density ratio $R_{\vr}$ (fixed gas parameter $a$ and pendulum length $L$) in Figure~\ref{fig:test1};
\item smaller gas parameter $a$ (fixed density ratio $R_{\vr}$ and pendulum length $L$) in Figure~\ref{fig:test2};
\item smaller pendulum length $L$ (fixed density ratio $R_{\vr}$ and gas parameter $a$) in Figure~\ref{fig:test3}.
\end{enumerate}





\begin{figure}[!h]
    \centering
\includegraphics[width=\textwidth]{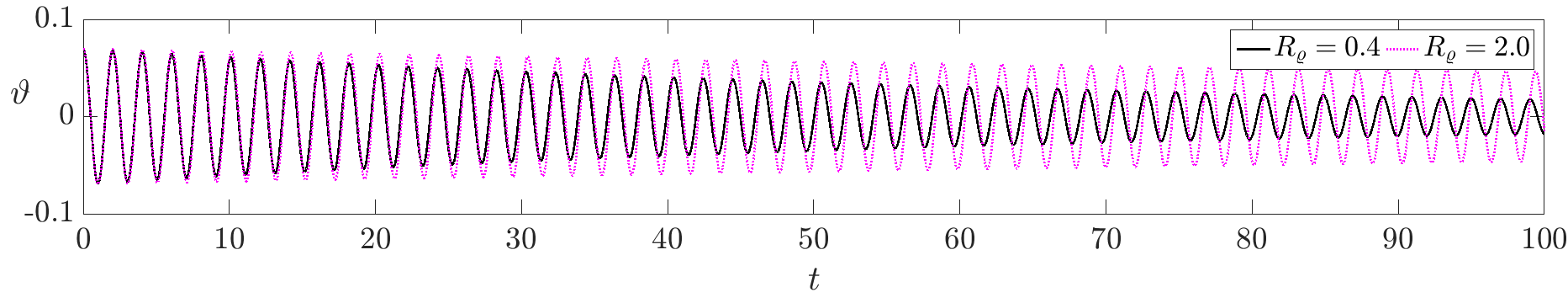}
    \caption{Evolution of pendulum position (angle $\vartheta$) for different density ratio $R_{\vr}$ with fixed gas parameter $a=10$ and pendulum length $L=0.4$.}
    \label{fig:test1}
\end{figure}

\begin{figure}[!h]
\centering
\includegraphics[width=\textwidth]{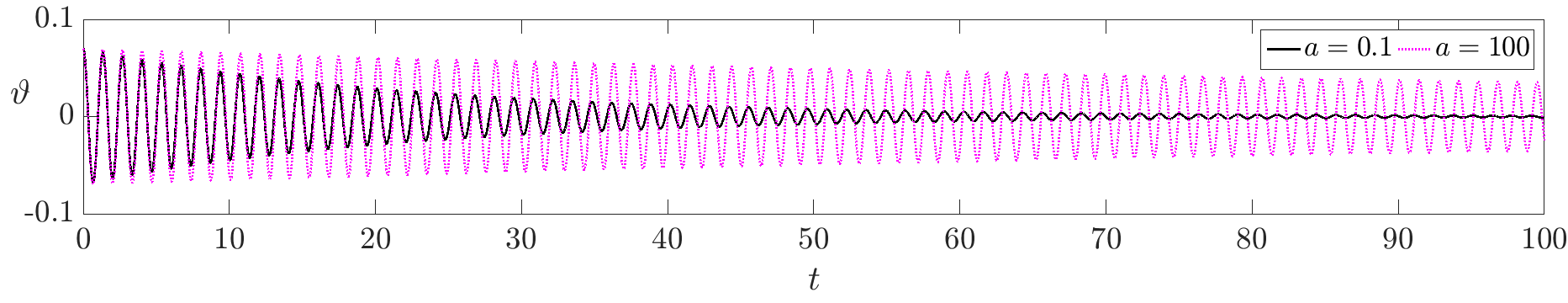}
    \caption{Evolution of pendulum position (angle $\vartheta$) for different gas parameter $a$ with fixed density ratio $r=1$ and length $L=0.4$.}
    \label{fig:test2}
\end{figure}
\begin{figure}[!h]
\centering
\includegraphics[width=\textwidth]{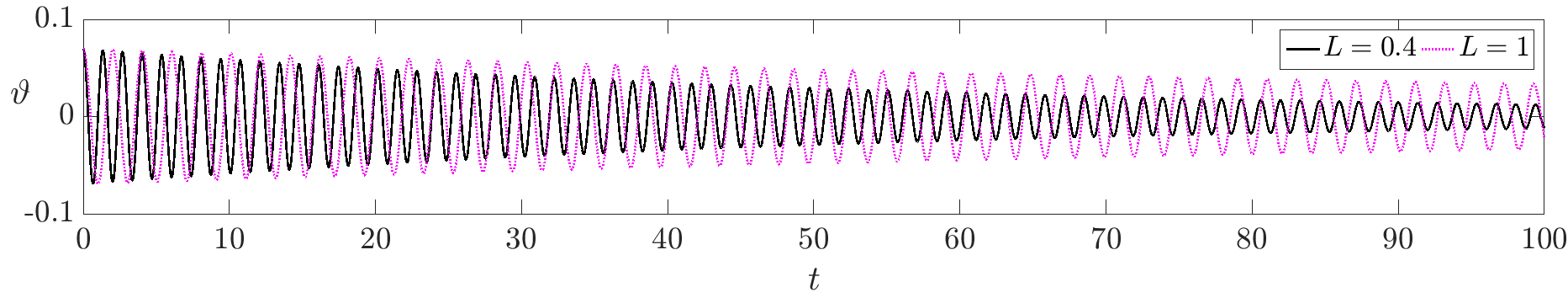}
    \caption{Evolution of pendulum position (angle $\vartheta$) for different pendulum length $L$ with fixed density ratio $r=1$ and gas parameter $a=10$.}
    \label{fig:test3}
\end{figure}

\subsubsection{Experiment 2: comparison with an incompressible solver}
To compare the damping effects of compressible and incompressible fluids, we also introduce an incompressible solver which replaces the Navier--Stokes part of the compressible solver \eqref{scheme}, that is \eqref{scheme_r}--\eqref{scheme_u}, by it incompressible counterpart, while keeping the method of $\vgh$ and $\vwh$ unchanged. 
\paragraph{Incompressible solver.} 
Let $\vuh^0,\vgh^0, \vwh^0$ be given in the same way as the compressible solver. 
Let $\vr_\calc$ be the density of the incompressible fluid in the cavity $\calc$. 
For $k=1,\ldots, N_t$ we seek $(p_h^k, \vuh^k, \vgh^k, \vwh^k) \in X_h \times \Vh \times \mathbb R^2 \times \mathbb R$ such that the following system of algebraic equations hold:
\begin{subequations}
\begin{equation*}
    \Dt \vgh^k  +  \vwh^{k-1} \ve_3 \times \vgh^{k+1/2} =0,
\end{equation*}
\begin{equation*}
\begin{aligned}
& \vr_\calc \intC{ \Big( 
   \Dt \vuh^k \cdot \bfphi
   +\frac12\bfvh^{k-1}\cdot \nabla\vuh^k \cdot \bfphi 
   -\frac12\bfvh^{k-1}\cdot \nabla \bfphi \cdot \vuh^k 
   \Big) } 
+ \vr_\calc \intC{ w_h^{k-1} \ve_3\times\vuh^k \cdot \bfphi }
\\& 
+ \intC{ \Big( \mathbb S(\vuh^k): \nabla \bfphi 
  - p_h^k \diver \bfphi-  q_h \diver \vuh^k  \Big)} 
  = \intC{\vr_\calc \vgh^{k+1/2} \cdot \bfphi},
 \quad \forall \; q_h \in  X_h,\;  \bfphi \in \Vhz ,
\end{aligned}
\end{equation*}

\begin{equation*}
  I_{B33} \Dt \vwh^k 
  + \Dt \left(\intC{\vr_\calc \bfx \times \vuh^k} \right) \cdot \ve_3 
 = \left( \intB{\varrho_{\mathscr B} \bfx }+\intC{\vr_\calc \bfx } \right)\times \vgh^{k+1/2} \cdot \ve_3,
 \end{equation*}
 where $X_h :=\{v\in L^2(\calc) |\;  v_K \text{ is a constant } \; \forall\; K \in \calah; \intC{v}=0 \}$.
\end{subequations}

We show in Figure~\ref{fig:testIn} the evolution of pendulum positions (represented by the angle $\vartheta$) obtained by the compressible solver and the incompressible solver. Here we have used \emph{same parameters} for both solvers: $L=0.4$, $\mu=100$, $\eta=0$, $\varrho_{\mathscr B}=1$, and initial fluid density $\vr_\calc=1.0$. Note that the only difference relies on the  gas parameter $a(=0.1,\; 20,\; 100)$ in the compressible solver, which is not needed in the incompressible solver. Here, let us point out that larger gas parameter $a$ means smaller Mach number. Obviously, Figure~\ref{fig:testIn} tells that compressible fluids brings more damping. 
\begin{figure}[!h]
\includegraphics[width=\textwidth]{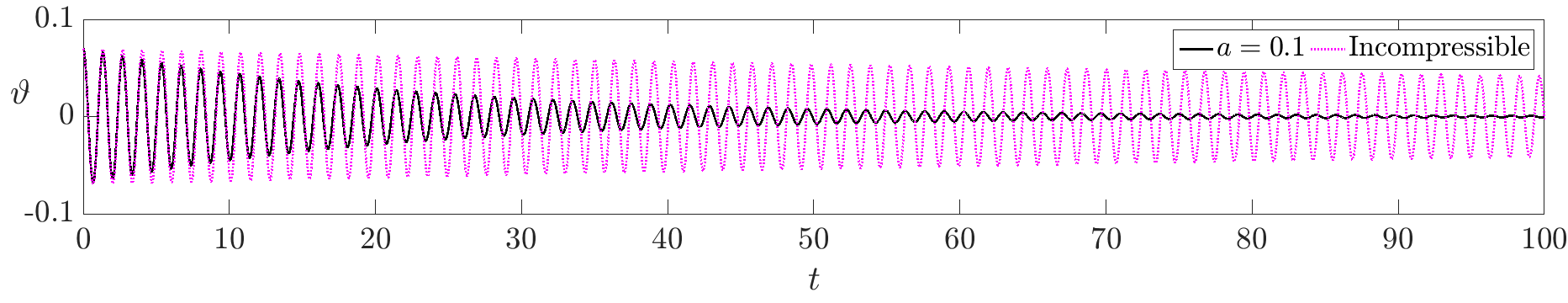}
\includegraphics[width=\textwidth]{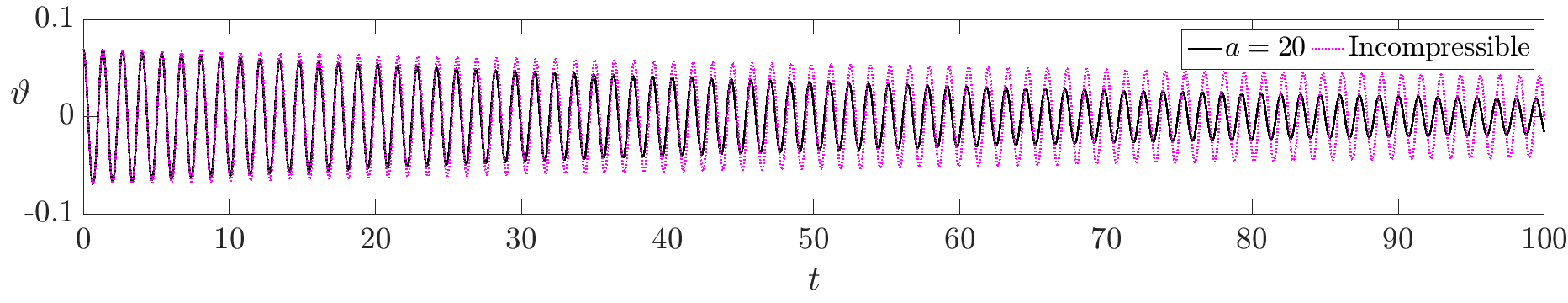}
\includegraphics[width=\textwidth]{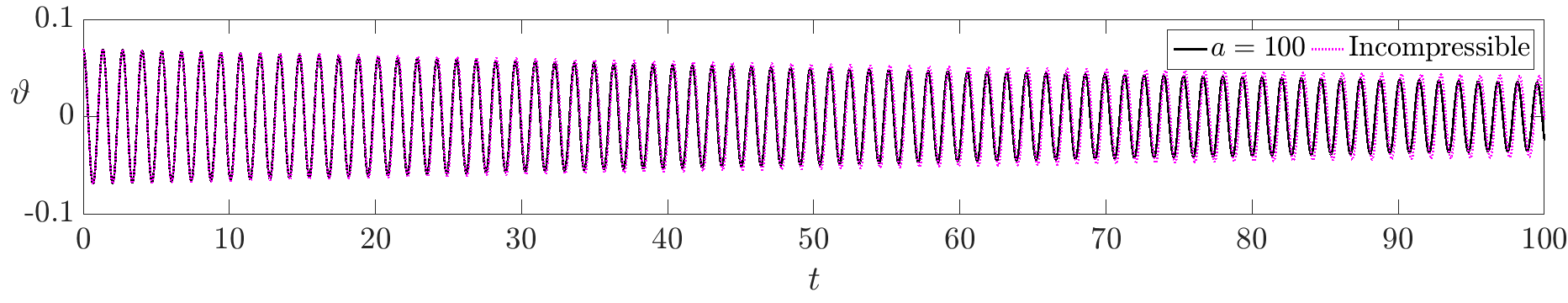}
\caption{Comparison of numerical solutions of compressible solver ($a=0.1$  $a=20$ and $a=100$) and incompressible solver.}
\label{fig:testIn}
\end{figure}

\section*{\centering Acknowledge}
The work of V\'aclav M\'acha and \v{S}\'arka Ne\v{c}asov\'a are supported by Praemium Academi\ae \ of \v S. Ne\v casov\'a and by grant GA\v CR GA22-01591S in the framework of RVO:67985840. 
The work of B.S. were supported by ERC-CZ grant LL2105 of the Ministry of Education, Youth and Sport of the Czech Republic and 
Charles University Research program No. UNCE/SCI/023

\section*{\centering Statements and Declarations}
{\bf Conflict of interests:}  On behalf of all authors, the corresponding author states that there is no conflict of interest.\\
{\bf Data  availability:} All data generated or analysed during this study are included in this article.

\newpage

\bibliographystyle{plain}

\begin{thebibliography}{}

\end{thebibliography}


\begin{thebibliography}{99}
\bibitem{AlSp}{\sc Bhuta, P.G.,  Koval, L.R.},   A viscous ring damper for a freely precessing
satellite. {\em Intern. J. Mech. Sci.} {\bf 8} (1966) 383--395
\bibitem{CaLeMo} {\sc Carl, S., Le, V., K., Montreanu, K.:} {\em Nonsmooth variational problems and their inequalities, Comparison principles and applications}. Springer, 2007 
\bibitem{Ce0} {\sc Chernousko, F.L.},{\em  Motion of a rigid body with cavities containing a viscous fluid}
(1968). NASA Technical Translations, Moscow, 1972
\bibitem{Ce}{\sc Chernousko, F.L., Akulenko, L.D., Leshchenko, D.D.}, {\em Evolution of Motions of a
Rigid Body About its Center of Mass}, Springer, Berlin 2017
\bibitem{DGMZ}{\sc Disser, K., Galdi, G.P., Mazzone, G., Zunino, P.}: Inertial Motions of a Rigid Body
with a cavity filled with a viscous liquid. {\em Arch. Ration. Mech. Anal}. {\bf 221}, 487--526
2016
\bibitem{EUS}{\sc Eswaran, M.,   Saha, U.K.}, Sloshing of liquids in partially filled tanks -- a review of experimental investigations, {\em Ocean Systems Engineering}, {\bf 1} (2011) 131--155
\bibitem{FeNo} {\sc Feireisl, E., Novotn\'y, A.:} {\em Singular limits in thermodynamics of viscous fluids}. Advances in Mathematical Fluid Mechanics.  Birkh\"auser, Basel, second edition, 2017

\bibitem{FeNoPe} {\sc Feireisl, E., Novotn\'y, A., Petzeltov\'a, H.:} {\em On the existence of globally defined weak solutions to the Navier-Stokes equations}. J. Math. Fluid Mech. 3 (2001), no. 4, 358--392

\bibitem{FePe} {\sc Feireisl, E., Petzeltov\'a, H.:} {\em Large-time behaviour of solutions to the Navier-Stokes equations of compressible flow}. Arch. Rational Mech. Anal 150 (1999), 77--96

\bibitem{FePe2} {\sc Feireisl, E., Petzeltov\'a, H.:} {\em On the zero-velocity-limit solutions to the Navier-Stokes equations of compressible flow}. Manuscripta Math. 97 (1998), no. 1, 109--116

\bibitem{FoLe} {\sc Fonseca, I., Leoni, G.:} {\em Modern methods in the calculus of variations: {$L^p$} spaces}. Springer, New York, 2007
\bibitem{GaIM}{\sc Galdi, G.P.}: Stability of permanent rotations and long-time behavior of inertial motions
of a rigid body with an interior liquid-filled cavity. Particles in flows, 217--253, {\em Adv.
Math. Fluid Mech.}, Birkhäuser/Springer, Cham, 2017
\bibitem{GaMaNe} {\sc Galdi, G., P., M\'acha, V., Ne\v casov\'a, \v S.:} {\em On the motion of a body with a cavity filled with compressible fluid}. Arch. Ration. Mech. Anal. 232 (2019), no. 3, 1649--1683
   
\bibitem{GaMaNe2} {\sc Galdi, G., P., M\'acha, V., Ne\v casov\'a, \v S.:} {\em On weak solutions to the problem of a rigid body with a cavity filled with a compressible fluid, and their asymptotic behavior}. International Journal of Non-Linear Mechanics 121, 2020, 103431
	
\bibitem{GaMa} {\sc Galdi, G., P., Mazzone, G.:} {On the motion of a pendulum with a cavity entirely filled with a viscous liquid} in {\em Recent progress in the theory of the Euler and Navier-Stokes equations}, 37--56, London Math. Soc. Lecture Note Ser., 430, Cambridge Univ. Press, Cambridge, 2016
\bibitem{GaMa1}{\sc  Galdi, G.P.; Mazzone, G.}: Nonlinear stability analysis of a spinning top with an interior liquid-filled cavity. {\em Math. Model. Nat. Phenom.} {\bf 16} (2021), Paper No. 22, 21 pp.	
\bibitem{GMZ}{\sc Galdi, G.P., Mazzone, G., Zunino P.}: Inertial motions of a rigid body with a cavity
filled with a viscous liquid, {\em Comptes Rendus M\'ecanique} {\bf 341}, 760--765 2013
\bibitem{GMM}{\sc Galdi, G.P., Mazzone, G., Mohebbi, M.}: On the motion of a liquid-filled rigid body
subject to a time-periodic torque, Recent developments of mathematical fluid mechanics,
{\em Adv. Math. Fluid Mech.}, Birkhäuser/Springer, Basel, pp. 233--255 2016
\bibitem{GMM1}{\sc Galdi, G.P., Mazzone, G., Mohebbi, M.}: On the motion of a liquid-filled heavy body
around a fixed point. {\em Q. Appl. Math. } {\bf 76}, 113--145 2018
\bibitem{NoSt} {\sc Novotn\'y, A., Stra\v skraba, I.:}{\em Introduction to the Mathematical Theory of Compressible Flow}, Oxford Lecture Series in Mathematics and its Applications 27, Oxford University Press, Oxford, 2004


\bibitem{Karlsen}
{\sc Karlsen K.~H., Karper T.~K.:} {\em Convergence of a mixed method for a semi-stationary compressible Stokes system}, Math. Comp. 80 (2011), 1459--1498 


\bibitem{KNN} {\sc Kra\v{c}mar, S., Ne\v{c}asov\'a, \v{S}., Novotn\'y, A.}, The motion of a compressible viscous fluid
around rotating body. {\em Ann. Univ. Ferrara} Sez. VII Sci. Mat. {\bf 60}, (2014) 189--208

\bibitem{kuku} {\sc Kuku\v cka, P.:} {\em On the existence of finite energy weak solutions to the Navier-Stokes equations in irregular domains},  Math. Methods Appl. Sci. 32 (2009)
\bibitem{Ma1}{\sc Mazzone, G.; Pr\"uss, J.; Simonett, G.}, A maximal regularity approach to the study of motion of a rigid body with a fluid-filled cavity. {\em J. Math. Fluid Mech.} {\bf 21} (2019),  Paper No. 44, 20 pp
\bibitem{Ma2} {\sc Mazzone, G.; Pr\"uss, J., Simonett, G.}, On the motion of a fluid-filled rigid body with Navier boundary conditions. SIAM {\em J. Math. Anal.} {\bf 51} (2019) 1582--1606

\bibitem{Ma3} {\sc Mazzone, G.}: On the free rotations of rigid bodies with a liquid-filled gap. {\em J. Math. Anal. Appl.} {\bf 496} (2021), no. 2, Paper No. 124826, 37 pp.
\bibitem{pedregal} {\sc Pedregal, P.:} {\em Parametrized measures and Variational principles}, Birkh\"auser Verlag,  Basel, 1997

\bibitem{tartar} {\sc Tartar, L.:} {\em Compensated compactness and applications to partial differential equations} in Nonlinear analysis and mechanics: Heriot--Watt Symposium, Vol. IV, pp. 136--212, Res. Notes in Math., 39, Pitman, Boston, Mass.--London, 1979. 
\bibitem{Tho}{\sc Thomson, W. (Lord Kelvin)}: On an experimental illustration of minimum energy. Nature. {\bf 23}, 69--70 (1880)


\end{thebibliography}

\end{document}